\documentclass{amsart}
\usepackage{amsfonts}
\usepackage{graphicx}
 \usepackage{epsfig}
 \usepackage[colorlinks=true]{hyperref}
\hypersetup{urlcolor=blue, citecolor=red}
\newtheorem{theorem}{Theorem}[section]
\newtheorem{lemma}[theorem]{Lemma}

\theoremstyle{definition}
\newtheorem{definition}[theorem]{Definition}
\newtheorem{proposition}[theorem]{Proposition}
\newtheorem{corollary}[theorem]{Corollary}
\newtheorem{example}[theorem]{Example}
\newtheorem{question}[theorem]{Question}

\theoremstyle{remark}
\newtheorem{remark}[theorem]{Remark}

\numberwithin{equation}{section}

\begin{document}

\title{the templates of nonsingular Smale flows on three manifolds}

\author{Bin Yu}
\address{Department of Mathematics, Tongji University, Shanghai, China 20092}

\thanks{The author was supported in part by the National NSF of China (grant no. 11001202).}

\subjclass[2000]{57N10, 58K05; 37E99, 37D45.}

\date{?????  ?, 2010 and, in revised form, ????? ?, 201?.}

\keywords{3-manifold, nonsingular Smale flow, template, filtrating
neighborhood}

\begin{abstract}
 In this paper, we first discuss some
 connections between template theory and the description of basic sets
of Smale flows on 3-manifolds due to F. B\'eguin and C. Bonatti. The
main tools we use are symbolic dynamics, template moves and some
combinatorial surgeries. Second, we obtain some relationship between
the surgeries and the number of $S^1 \times S^2$ factors of $M$ for
a nonsingular Smale flow on a given closed orientable 3-manifold
$M$. Besides these, we also prove that any template $T$ can model a
basic set $\Lambda$ of a nonsingular Smale flow on $nS^1 \times S^2$
for some positive integer $n$.
\end{abstract}
\maketitle

\section{Introduction}
The paper is about using template theory to understand nonsingular
Smale flows (abbreviated as NS flows) on 3-manifolds.

On the level of topological equivalence, dynamics of general
structually stable flow is complicated and not much is know, even in
dimension 3. However, if we restrict ourselves to a special class,
various methods apply. For example, nonsingular Morse Smale flows
(abbreviated as NMS flows) have been effectively studied by
associating NMS flows with a kind of combinational tool, i.e., round
handle decomposition, see \cite{As}, \cite{Mo} and \cite{Wa}. Here
an NMS flow is a structurally stable flow whose nonwandering set is
exactly composed of finite closed orbits.

A larger class than NMS flows  is NS flows. NS flows were first
introduced by J. Franks in the 1980s, see \cite{F1}, \cite{F2}, and
\cite{F5}. An NS flow is a structurally stable flow with
one-dimensional invariant sets and without singularities. For
general NS flows, the situation becomes more complicated than NMS
flows. A theorem of Bowen \cite{Bo} says that a one-dimensional
hyperbolic basic set $\Lambda$ must be topologically equivalent to a
suspension of a subshift of finite type (abbreviated as SSFT).
Therefore, symbolic dynamics is very useful in the study of
one-dimensional hyperbolic basic set. However, it doesn't provide
any embedding information. J. Franks (\cite{F1},  \cite{F2},
\cite{F5}) used homology and graphs to describe some embedding
information of NS flows on 3-manifolds.

To  study NS flows on 3-manifolds extensively, a natural idea is to
discuss a kind of neighborhoods of basic sets and then study how to
glue these neighborhoods together.

Template theory (see Section \ref{section2.1}) provides a kind of
neighborhoods of one-dimensional hyperbolic basic sets, i.e.,
thickened templates.  Thickened templates are useful to the study of
NS flows on 3-manifolds because:
\begin{enumerate}
\item A thickened template is homeomorphic to a
handlebody;
\item Using a thickened template, one can easily read the symbolic dynamics of a basic set and
describe the knot types of the closed orbits in the basic set.
\end{enumerate}
M. Sullivan \cite{Su} and the author \cite{Yu1} used thickened
templates to
 discuss a special type of NS flows on
 3-manifolds. However, it seems hopeless to establish a general framework for NS flows on
 3-manifolds by using thickened templates due to the following two facts:
\begin{enumerate}
\item It is difficult to describe the entrance set and the exit set in
the boundary of a thickened template;
\item Many different templates can be used to model a given nontrivial basic set.
\end{enumerate}

In \cite{BB}, F. B\'eguin and C. Bonatti gave several concepts and
results to describe the behavior of a given nontrivial basic set. An
introduction to their work can be found in Section \ref{section2.3}.
Their work provided another kind of neighborhoods of nontrivial
basic sets on 3-manifolds, i.e., filtrating neighborhoods. We note
that the same concept was also discovered by J. Franks in \cite{F5},
who named it ``building block". Filtrating neighborhoods are useful
because:
\begin{enumerate}
\item  The filtrating neighborhood
of $\Lambda$ is unique up to topological equivalence for a given
nontrivial basic set $\Lambda$;
\item Filtrating neighborhoods are natural chunks for the
reconstruction of the underlying 3-manifold with a Smale flow. See
\cite{F5} and \cite{Re}.
\end{enumerate}
In spite of these advantages, it is difficult to study the
topological structures of filtrating neighborhoods systematically.

In this paper, we first give a structure theorem  (see Theorem
\ref{theorem4.4}) to show some connections between the filtrating
neighborhood and a  thickened template for a given nontrivial basic
set. The structure is described by some parameters $(g_i,k_i,t)$. In
this process, we also discuss some connections between template
theory and some concepts proposed by F. B\'eguin and C. Bonatti (see
Theorem \ref{theorem3.3} and Theorem \ref{theorem4.3}). Our main
tools are some combinatorial surgeries, i.e., surgeries on edge
graphs, template moves and attaching thickened
 surfaces with flows.

 In Theorem \ref{theorem4.5}, we find that for an NS flow on a 3-manifold $M$,
 the parameters $(g_i,k_i)$ are bounded by a topological invariant of 3-manifolds, i.e.,
  the number of the $S^1 \times S^2$ factors
  of $M$ (see Theorem \ref{theorem4.5}). As a corollary (Corollary
  \ref{corollary4.6}), given a template $T$ and a closed orientable 3-manifold
  $M$, there exist at most finitely many filtrating neighborhoods (up to
topological equivalence)  modeled by $T$ in an NS flow on $M$.  As
an application, we give an example (Example \ref{example4.7}) to
discuss the realization of a special type of NS flows on
3-manifolds. Example \ref{example4.7} also shows that although the
topological type of filtrating neighborhood (given $T$ and $M$) is
finite, there are infinitely many possible different ways to embed
some filtrating neighborhood of $T$ as a filtrating neighborhood of
an NS flow on $M$.

 Another interesting problem is the realization of
 one-dimensional basic sets in NS flows on 3-manifolds, which will
 be the topic in the last section. If the basic sets are described by SSFT, the realization problem
is considered in \cite{PS} and
 \cite{F3}. The realization problem of SSFT is completely solved in the sense that for any SSFT and any
 closed orientable 3-manifold $M$,
  there is an NS flow $\phi_t$ on $M$ such that the SSFT can be realized as a basic set of $\phi_t$ (see
  \cite{F5} and Proposition 6.1 in \cite{F1}).
If the basic sets are described by a template,  Meleshuk \cite{Me}
proved that there exists some 3-manifold admitting an NS flow such
that there exists a basic set of the NS flow modeled by the
template.  It is natural to ask: for a given template $T$, which
3-manifolds admit (or do not admit) an NS flow with a basic set
modeled by $T$? The author \cite{Yu2}, proved that:
\begin{enumerate}
\item If a closed orientable 3-manifold $M$ admits an NS flow with a basic
set modeled
     by $T$, then $M=M'\sharp g(T) S^{1}\times S^{2}$. Here $g(T)$
     is the genus of $T$, see \cite{Yu2}.
\item There exists a closed orientable
     3-manifold $M'$  such that
     $M=M'\sharp g(T) S^{1}\times S^{2}$ admits an NS flow with a basic set
     modeled by $T$.
\end{enumerate}
In the last section, we prove that any
  template $T$ can model a basic set $\Lambda$ of an NS flow $\psi_t$
on $nS^1 \times S^2$ for some positive integer $n$ (obviously $n\geq
g(T)$), see Theorem \ref{theorem5.5}.

\section{Preliminaries} \label{section2}
Basic definitions and facts about dynamical systems can be found in
\cite{Ro}.

\subsection {Template} \label{section2.1}

A \textit{Template}  $(T,\phi)$ is a smooth branched 2-manifold $T$,
constructed from two types of charts, called \textit{joining charts}
and \textit{splitting charts}, together with a semi-flow. A
semi-flow is the same as flow except that one cannot back up
uniquely. The semi-flows are indicated by arrows on charts in Figure
1. The gluing maps between charts must respect the semi-flow and act
linearly on the edges.

\begin{figure}[htp]
\begin{center}
  \includegraphics[totalheight=4.5cm]{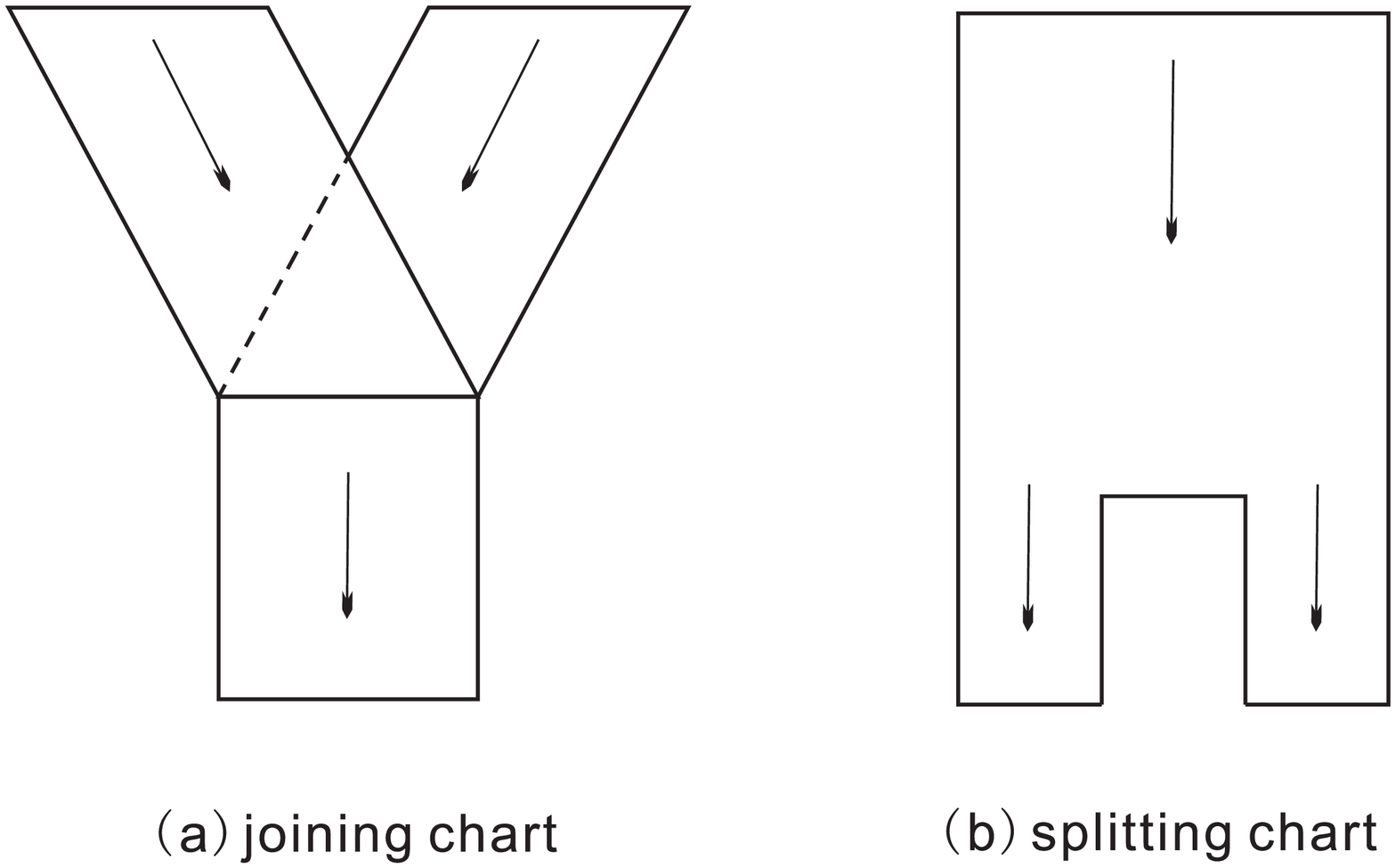}\\
  \caption{}\label{figure1}
  \end{center}
\end{figure}

The following template theorem of R. Williams and J. Birman
(\cite{BW1}, \cite{BW2}, \cite{GHS}) shows the importance of
template theory to the study of flows on 3-manifolds.

\begin{theorem}[the template theorem] \label{theorem2.1}
Let $\phi_{t}$ be a smooth flow on a 3-manifold M with a hyperbolic
basic set. The link of closed orbits $L_{\phi}$ is in
 bijective correspondence with the link of closed orbits $L_{T}$
 on a particular embedded template $T\subset M$ (with $L_{T}$ containing at
 most two extraneous orbits).
\end{theorem}

\begin{remark}
In Theorem 2.1, if the topological dimension of the basic set is 1,
then the correspondence is exactly bijective without any extraneous
orbits.
\end{remark}

Now we introduce template moves, which were introduced by M.
Sullivan and the others (\cite{GHS}, \cite{KSS}).

\begin{definition}[template moves] \label{definition2.3}
For a template, as Figure \ref{figure2} shows, the move in Figure
\ref{figure2} (a) is called a \emph{slide move};  the move in Figure
\ref{figure2} (b) is called a \emph{split move}. Slide moves, split
moves and their converse moves are collectively referred to as
\emph{template moves}.
\end{definition}

\begin{figure}[htp]
\begin{center}
  \includegraphics[totalheight=3cm]{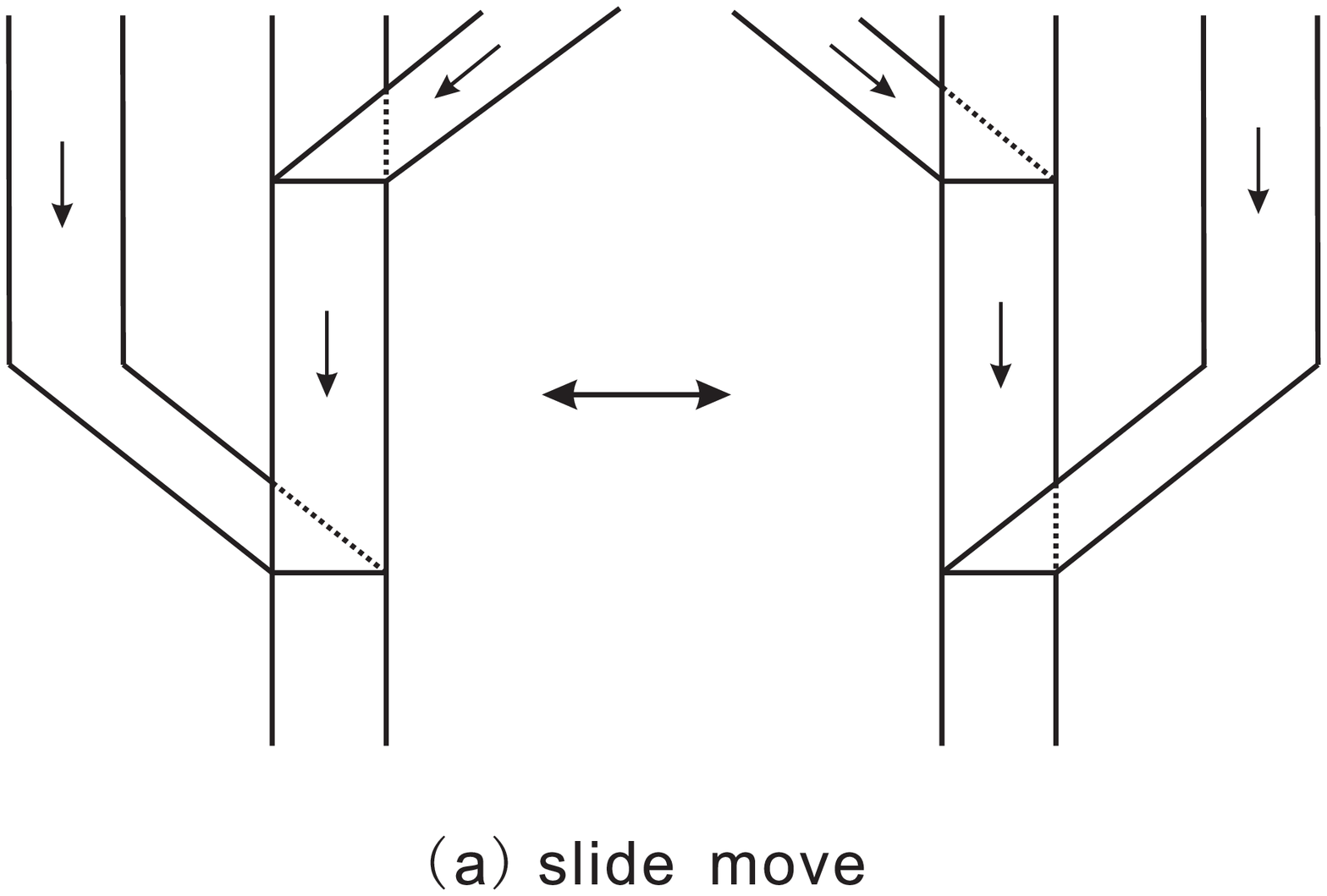}\\
  \end{center}
\end{figure}

\begin{figure}[htp]
\begin{center}
  \includegraphics[totalheight=3cm]{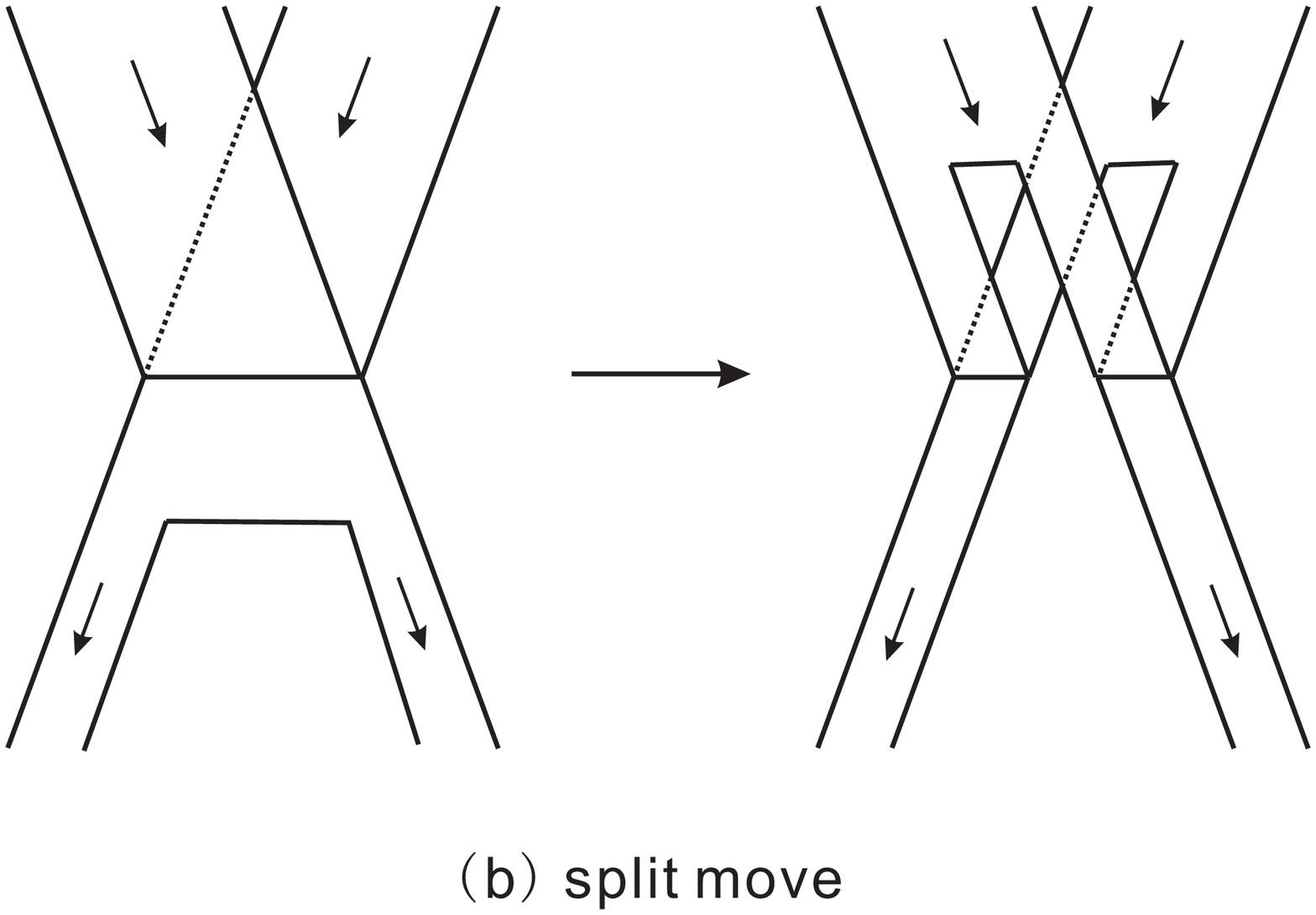}\\
  \caption{}\label{figure2}
  \end{center}
\end{figure}

We now briefly recall the construction of the templates of dimension
1 hyperbolic basic sets, which is useful to our discussion. More
details can be found in
 \cite{GHS}.

Let $B$ be a dimension 1  hyperbolic basic set of a flow $\phi_{t}$.
By Theorem \ref{theorem2.8}, $B$ is conjugate to an SSFT.
Furthermore, Bowen and Walters \cite{BoW} showed that there exists
 a finite union of
disjoint discs such that the Poincar\'e return map on these discs
satisfies the condition of Markov partitions. Throughout this paper,
this union of disjoint discs is called a cross-section. The
neighborhood of one of the above discs is known as a \textit{Markov
flowbox neighborhood}. The incoming and outgoing flowboxes near a
flowbox are shown in Figure \ref{figure4}(a). A flowbox neighborhood
is used to model the flow on the neighborhood of one of the above
discs. With this in mind, we can normalize it to Figure
\ref{figure4}(b) through a small perturbation.

Further, we decompose the perturbed Markov flowbox into some
\textit{joining flowboxes} and some \textit{splitting flowboxes} in
Figure \ref{figure5}. The gluing maps between simple flowboxes must
respect the flow and act linearly on the boundaries.

For a given Markov flowbox neighborhood of $B$ constructed above, we
can ``crush" the stable foliations to $T$ which is a branched
manifold with semiflow. Locally, we obtain joining (resp. splitting)
charts from joining (resp. splitting) flowboxes and gluing maps
between charts from gluing maps between simple flowboxes. It follows
naturally from the definition that the gluing maps between charts
respect the flow and act linearly on the edges.

\begin{figure}[htp]
\begin{center}
  \includegraphics[totalheight=4cm]{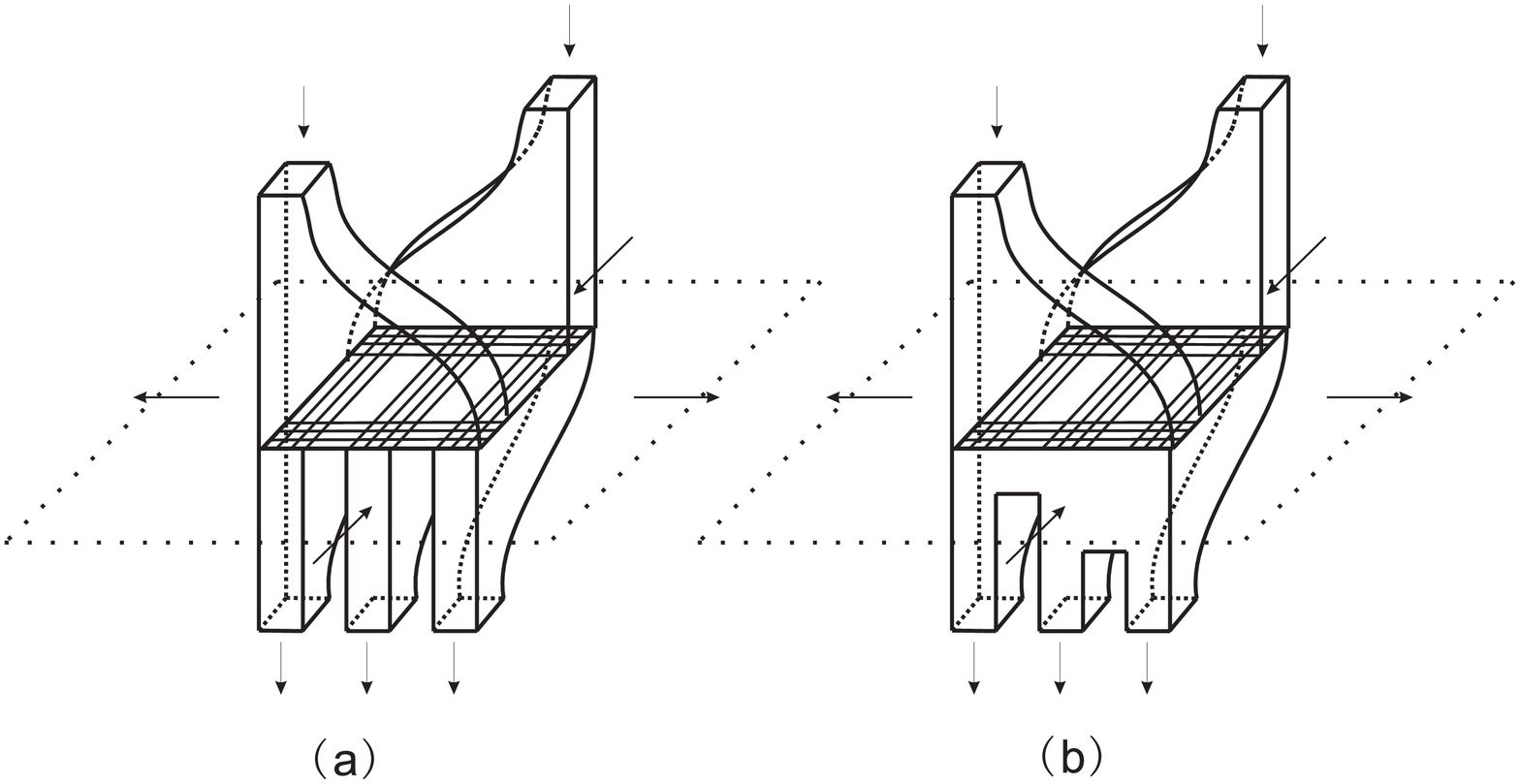}\\
  \caption{}\label{figure4}
  \end{center}
\end{figure}

\begin{figure}[htp]
\begin{center}
  \includegraphics[totalheight=4cm]{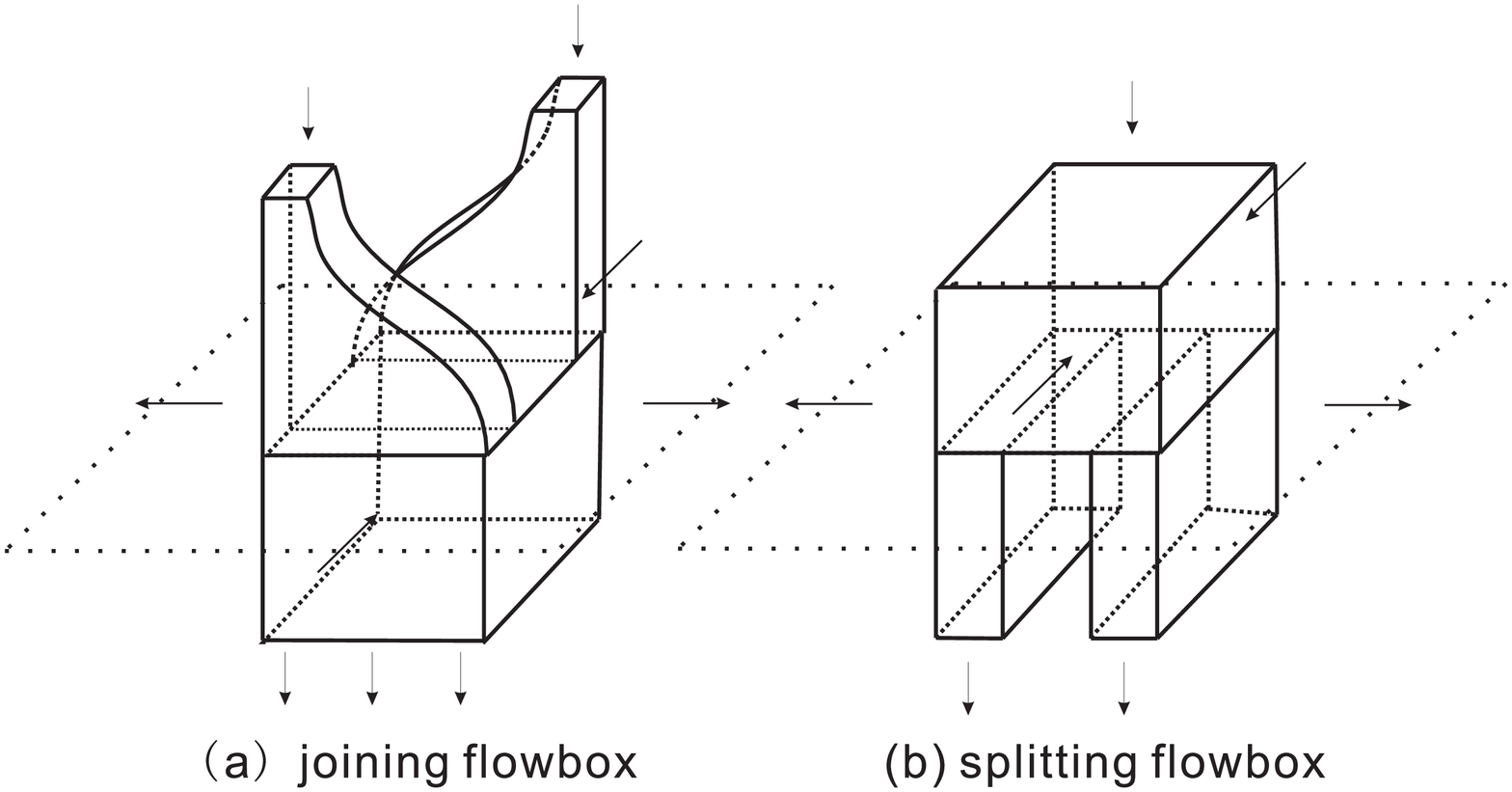}\\
  \caption{}\label{figure5}
  \end{center}
\end{figure}

Therefore, we obtain an embedded template which models the dimension
1 basic set.

\begin{remark}
A template move changes the topological type of the template but
that it is done in such a way as to not change the basic set.
\end{remark}

\subsection{Symbolic dynamics and State splitting of edge graphs}
All the definitions and the facts in this subsection can be found in
\cite{F2}, \cite{LM}.

Let $\Sigma_{n} = \{s=(...,s_{-1},s_{0},s_{1},...) \mid s_{j}\in
\{1,...,n\},\forall j\in \mathbb{Z}\}$. We can define a distance on
$\Sigma_{n}$ as follows: for any $s,t\in\Sigma_{n}$,
$d(s,t)=\Sigma_{j=-\infty}^{+\infty}
\frac{\delta(s_{j},t_{j})}{4^{|j|}}$ where  $\delta(a,b)=\{_{1,~if
~a\neq b}^{0, ~if~a=b}$. $\Sigma_{n}$ is called a \textit{two sided
shift space}. A homeomorphism $\sigma$ on $\Sigma_{n}$ is defined as
following. For any $s\in\Sigma_{n}$, let $t=\sigma(s)$ where
$t_{j}=s_{j+1}$. $(\Sigma_{n},\sigma)$ is called a \textit{full two
sided shift}.

Let $A=(a_{ij})_{n\times n}$, $a_{ij}\in \{0,1\}$. $A$ is called a
\textit{transition matrix}. $\Sigma_{A}$ is defined by
$\Sigma_{A}=\{s=(s_{j})\in \Sigma_{n} \mid a_{s_{j},s_{j+1}}=1,
\forall j\in \mathbb{Z}\}$. We denote $\sigma|_{\sum_{A}}$ by
$\sigma_{A}$. $(\Sigma_{A}, \sigma_{A})$ is called a
\textit{subshift of finite type}.

Actually, we can define a subshift of finite type for any matrix
$A=(a_{ij})_{n\times n}, a_{ij}\in \mathbb{Z}^{+}$ (such a matrix is
called an \textit{adjacent matrix}). For this purpose, we need to
introduce two concepts, \textit{vertex graphs} and \textit{edge
graphs}. For an adjacent matrix $A$, we take $n$ vertexes
$v_{1},v_{2},...,v_{n}$, then attach  $a_{ij}$ oriented edges from
$v_{i}$ to $v_{j}$ for any $i,j\in \{1,...,n\}$. Therefore, we get
an oriented graph $G$ which is called \textit{the vertex graph} of
$A$. Denote the edges of the vertex graph $G$ by $e_{1}, e_{2},...,
e_{m}$. For \emph{the edge graph} $F$ of $A$, take the edges of $G$
as the vertexes, which are still denoted by $e_{1}, e_{2},...,
e_{m}$ and join $e_i$ to $e_j$ if there exists a vertex $p$ in $G$
such that $e_i$ starts from $p$ and $e_j$ terminates at $p$.
Obviously, for two given vertexes $e_{i},e_{j}$ in $F$, there exists
at most one edge which starts from $e_{i}$ and terminates at
$e_{j}$. $F$ determines a unique transition matrix $A'$. $A'$
determines a subshift of finite type, which is called the
\textit{edge shift} of $A$.

For two adjacent matrixes $A$ and $B$, let $X_{A}$ and $X_{B}$ be
the edge shifts of $A$ and $B$ respectively. It is natural to ask
when $X_{A}$ is conjugate to $X_{B}$. In \cite{Wi1}, R. Williams
gave a necessary and sufficient condition by using adjacent matrix.
We state this condition by using edge graphs, see \cite{LM}.

\begin{definition}
\begin{enumerate}
\item Let $G$ be an edge graph with vertex set $V$ and edge set
$\varepsilon$. For a vertex $v\in V$, we divide $v$ into $v_{1}$ and
$v_{2}$, split outgoing edges of $v$ and copy incoming edges of $v$
to obtain a new graph. This surgery is called an \textit{out-split},
see Figure \ref{figure3}. Reversing the orientation of $G$, we
obtain a new graph $G_{1}$. Doing an out-split surgery on $G_{1}$,
we obtain a new graph $G_{2}$. Then reversing the orientation of
$G_{2}$, we obtain a new graph $G'$. We call the surgery from $G$ to
$G'$ an \textit{in-split}.
\item The converse surgery of an
out-split (resp. in-split) is called an \emph{out-amalgamation}
(resp. \textit{in-amalgamation}).
\item By  dividing an
 oriented edge of an edge graph $G$ to two oriented (the same orientation) edges
 by increasing a vertex to the oriented edge of $G$, we obtain a new
 graph $G'$. This surgery is called an \textit{expanding surgery}.
\end{enumerate}
\end{definition}
\begin{figure}[htp]
\begin{center}
  \includegraphics[totalheight=3cm]{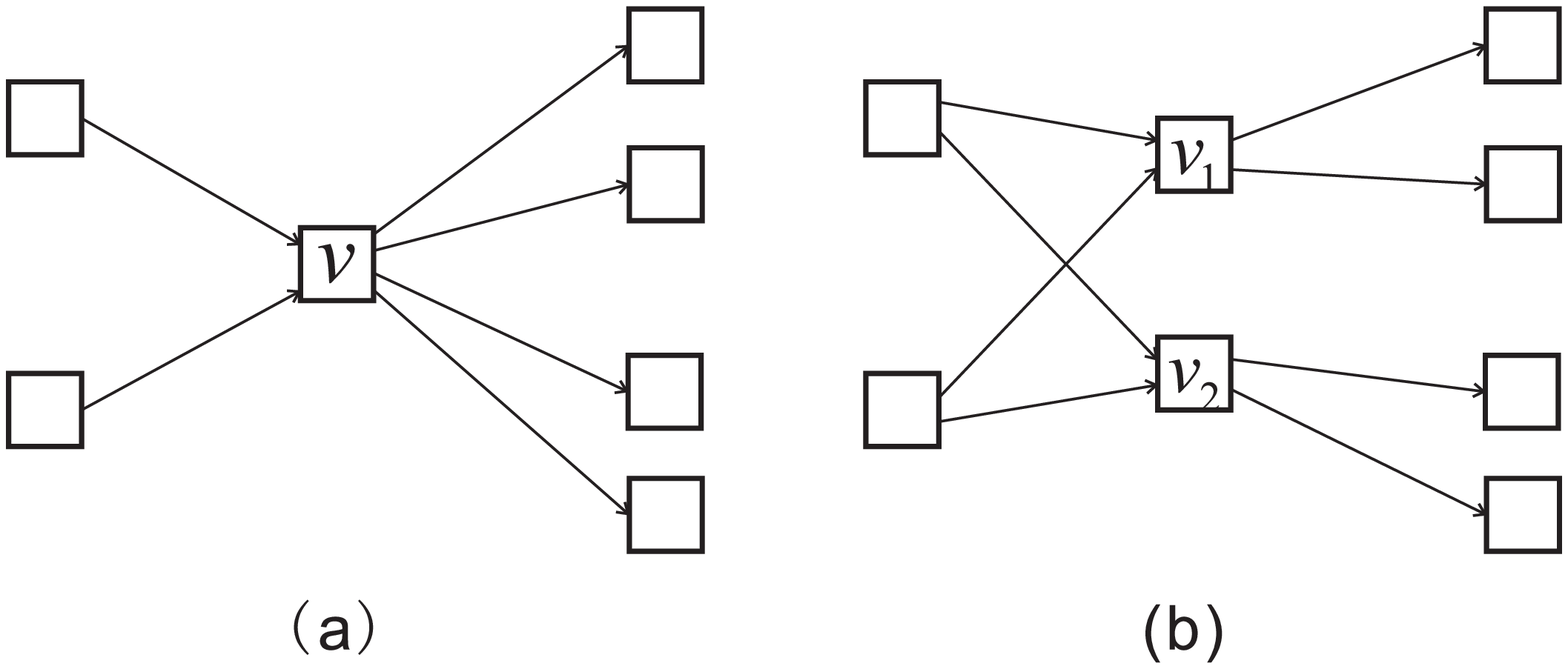}\\
  \caption{}\label{figure3}
  \end{center}
\end{figure}

\begin{theorem}
Every conjugacy from one edge shift to another is a composition of
finite steps of splitting surgeries and amalgamation surgeries.
\end{theorem}

For SSFT, Parry and D. Sullivan \cite{PaS} gave a complete
classification
  up to topological equivalence
 by using adjacent matrix. Moreover, J. Franks \cite{F4} found a complete set of computable
 algebraic invariants.  Here we only state this kind of result by using edge graph,
 see \cite{LM}.

\begin{theorem}\label{theorem2.6}
Two SSFT given by edge graphs $G$ and $F$ are topologically
equivalent if and only if there exists a finite sequence of edge
graphs $G=G_{0},G_{1},...,G_{r}=F$ satisfying $G_{i}\sim_{s}G_{i+1}$
or $G_{i}\sim_{e}G_{i+1}$ for $i=0,...,r-1$.
\end{theorem}

 Here $G_{i}\sim_{s}G_{i+1}$ means that $G_{i+1}$ can be obtained from $G_{i}$ by a splitting or
 an amalgamation. $G_{i}\sim_{e}G_{i+1}$ means that $G_{i+1}$ can be obtained from $G_{i}$ by an
expanding surgery or the converse of an expanding surgery.

  \subsection{NS flows on three manifolds} \label{section2.3}

  \begin{definition}
 A smooth flow $\phi_{t}$ on
a compact manifold $M$ is called a \textit{Smale flow} if:
\begin{enumerate}

\item the chain recurrent set $R(\phi_{t})$ has hyperbolic structure;

\item  $\dim (R(\phi_{t}))\leq 1$;

\item $\phi_{t}$ satisfies the transverse condition.
\end{enumerate}

  If a Smale flow $\phi_{t}$ has no singularity, we call
  $\phi_{t}$ a \textit{nonsingular Smale flow} (an \textit{NS flow}).
  In particular, if $R(\phi_{t})$ consists entirely of closed
  orbits, $\phi_{t}$ is called a \textit{nonsingular Morse Smale flow}
  (an \textit{NMS flow}).
\end{definition}

The following theorem is due to Smale \cite{Sm} and Bowen \cite{Bo}.

\begin{theorem} \label{theorem2.8}
The chain recurrent set $R(\phi_{t})$ of an NS flow $\phi_{t}$ on a
compact manifold $M$ satisfies:
\begin{enumerate}

\item  (Spectral decomposition) $R(\phi_{t})=\Lambda_1\sqcup ... \sqcup \Lambda_n$.
Here each $\Lambda_i$ ($i=1,..,n$) is called a \textit{basic set},
    i.e., $\Lambda_i$ is an invariant subset of $R(\phi_{t})$ and is a closure of an orbit;

 \item  $\Lambda_i$ is a closed orbit or $\Lambda_i$ satisfies that $\phi_{t}$ restricted to $\Lambda_i$ is topologically equivalent to
an SSFT.
\end{enumerate}
\end{theorem}

Now we introduce the descriptions of basic sets of Smale flows on
3-manifolds by F. Beguin and C. Bonatti \cite{BB}.

Let $\phi_{t}$ and $\varphi_{t}$ be two smooth flows with compact
invariant sets $K$ and $L$ respectively.
 $(\phi_{t},K)$ and $(\varphi_{t},L)$ are said to be equivalent
 if and only if there exists a neighborhood $U$ of $K$ and a
neighborhood $V$ of $L$ such that the restrictions of $\phi_{t}$ to
$U$ and $\varphi_{t}$ to $V$ are topologically equivalent via a
homeomorphism sending $K$ to $L$. The \textit{germ} $[\phi_{t},K]$
is the equivalence class represented by $(\phi_{t},K)$.

Let $\phi_{t}$ be a smooth flow on a closed orientable 3-manifold
$M$ and let $K$ be a saddle set of $\phi_{t}$. By a \emph{model} of
the germ  $[\phi_{t},K]$, we mean a pair $(\varphi_{t},N)$,  where
$N$ is a compact orientable 3-manifold and $\varphi_{t}$ is a smooth
flow on $N$ transverse to the boundary, such that:
\begin{enumerate}
\item The maximal invariant set for $\varphi_{t}$ in $N$ is a saddle
set $K_{\varphi_{t}}$ such that the germ
$[\varphi_{t},K_{\varphi_{t}}]$ is equal to the germ $[\phi_{t},K]$.
\item Denote by $\partial_{1} N$ the union of the connected
components of the boundary of $N$ where $\varphi_{t}$ is coming into
$N$; then any circle embedded in $\partial_{1} N$ and disjoint with
$W^{s}(K_{\varphi_{t}})$ bounds a disc in $\partial_{1} N$ which is
also disjoint with $W^{s}(K_{\varphi_{t}})$.
\item Any connected component of $N$ contains at least one point of
$K_{\varphi_{t}}$.
\end{enumerate}

\begin{theorem}
Given a saddle basic set $K$ of an NS flow $\phi_{t}$ on a closed
orientable 3-manifold $M$, there exists a unique (up to topological
equivalence) model $(\varphi_{t},N)$ of the germ $[\phi_{t},K]$.
\end{theorem}

\begin{theorem}
Let $\phi_{t}$ and $\varphi_{t}$ be two Smale flows on two
orientable 3-manifolds $M$ and $N$. Let $K$ and $L$ be two saddle
basic sets of $\phi_{t}$ and $\varphi_{t}$ respectively. If the
models of $[\phi_{t},K]$ and $[\varphi_{t},L]$ are topologically
equivalent, then there exist invariant neighborhoods $U$ of $K$ and
$V$ of $L$ such that the restriction of $\phi_{t}$ to $U$ is
topologically equivalent to the restriction of $\varphi_{t}$ to $V$.
\end{theorem}

A filtrating neighborhood of $K$ is a neighborhood $U$ such that:
\begin{enumerate}

\item $K$ is the maximal invariant set (for $\phi_{t}$) in $U$.
 \item The intersection of any orbit of $\phi_{t}$ with $U$ is
 connected (in other terms, any orbit getting out of $U$ never comes
 back).
\end{enumerate}

Let $(\varphi_{t},N)$ be a model of the germ $[\phi_{t},K]$. Let
$D_{1}$ and $D_{2}$ be two closed disjoint discs contained in the
entrance boundary of $N$ such that the orbits of $D_{1}$ and $D_{2}$
leave $N$ in a finite amount of time. The orbits of these two discs
form the union of two cylinders $D^{2}\times [0,1]$ endowed with the
vector-field $\frac{\partial}{\partial t}$. Let us cut out of $N$
these two cylinders and then paste the two resulting tangent
boundary components. By doing so, we have added to $N$ a handle,
getting a new 3-manifold still endowed with a vector-field. This
surgery is called \textit{handle attachment}.

\begin{proposition}\label{proposition2.11}
 If $U$ is
a filtrating neighborhood of $K$ such that all connected components
of $U$ meet $K$, one can obtain a pair topologically equivalent to
$(U,K)$ by modifying $(\varphi_{t},N)$ by a finite number of handle
attachments of the above type.
\end{proposition}




\section{Template and Germ} \label{section3}
If we extend a template $T$ in the direction perpendicular to its
surface, we obtain a thickened template  $\overline{T}$. The
semi-flow on $T$ extends to a flow on $\overline{T}$, see Figure
\ref{figure6}.  $\partial \overline{T}$ is composed of the entrance
set $X$, the exit set $Y$ and the dividing curves set $C$. More
details about thickened templates can be found in \cite{Me}. Denote
$\varphi^{T}_{t}$ by the flow on $\overline{T}$ and $K_{T}$ by the
invariant set of $\varphi^{T}_{t}$.
\begin{figure}[htp]
\begin{center}
  \includegraphics[totalheight=2.5cm]{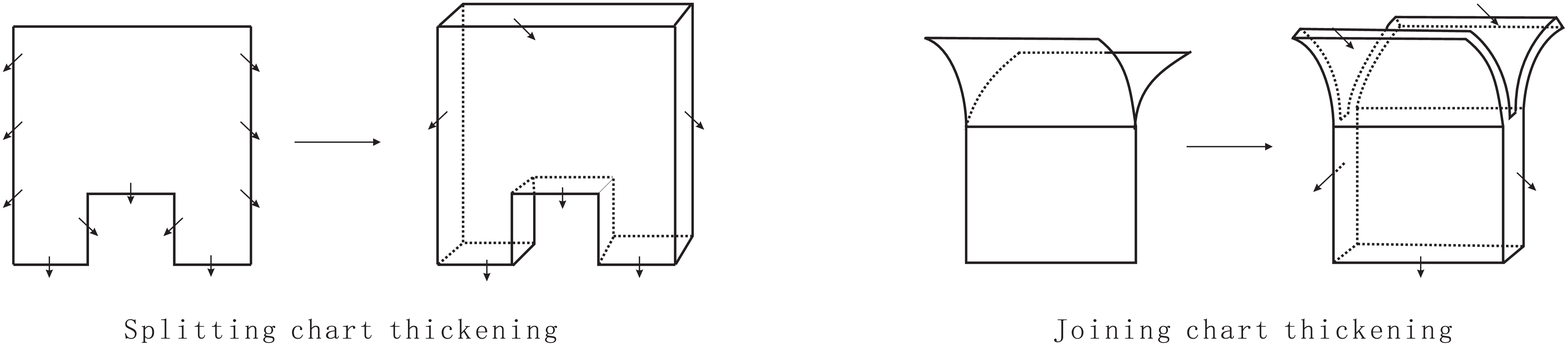}\\
  \caption{}\label{figure6}
  \end{center}
\end{figure}

 Then, we can define a germ $[\varphi^{T}_{t}, K_{T}]$ for a
 template $T$. On the other hand, by Theorem \ref{theorem2.1} (the template theorem), any germ can be
 represented by a template.

 By the proof of Theorem 2.5 in \cite{PaS}, we have the
following lemma.

\begin{lemma} \label{lemma3.1}
Up to isotopic equivalence, different cross-sections can be
exchanged by a finite sequence of unstable direction divisions (or
their converses), stable direction divisions (or their converses)
and adding parallel cross sections (or their converses).
\end{lemma}

Now we state the key lemma in this section.

\begin{lemma} \label{lemma3.2}
Let $\phi_{t}$ be a smooth flow on a 3-manifold M having a
hyperbolic chain-recurrent set which contains a dimension 1 basic
set $\Lambda$. Every two templates induced by $\Lambda$ can be
exchanged by a finite sequence of template moves.
\end{lemma}
\begin{proof}
Let $T_{1}$ and $T_{2}$ be two templates embedded in $M$ such that
both of them model $\Lambda$. Obviously, there exist two
cross-sections of $\Lambda$ denoted by  $\mathcal {C}_1$ and
$\mathcal{C}_2$ corresponding to $T_1$ and $T_2$ respectively. If we
crush the unstable
 foliation of $T_1$ and $T_2$, we can obtain two oriented embedded graphs $G_1$ and $G_2$.
 It is easy to see that the two SSFT of the edge shifts defined by both $G_1$ and $G_2$ are topologically equivalent to
$\phi_{t}|_{\Lambda}$. Hence $G_1$ and $G_2$ can be regarded as two
edge graphs of
 $\phi_{t}|_{\Lambda}$.
Given the relation between $G_i$ and $T_i$, we call $G_i$ the edge
graph of $T_i$ for $i=1,2$. $G_1$ and $G_2$ are the so called
\emph{template edge graphs} in the sense that they can be obtained
by crushing the unstable foliations of two templates.

\begin{figure}[htp]
\begin{center}
  \includegraphics[totalheight=3cm]{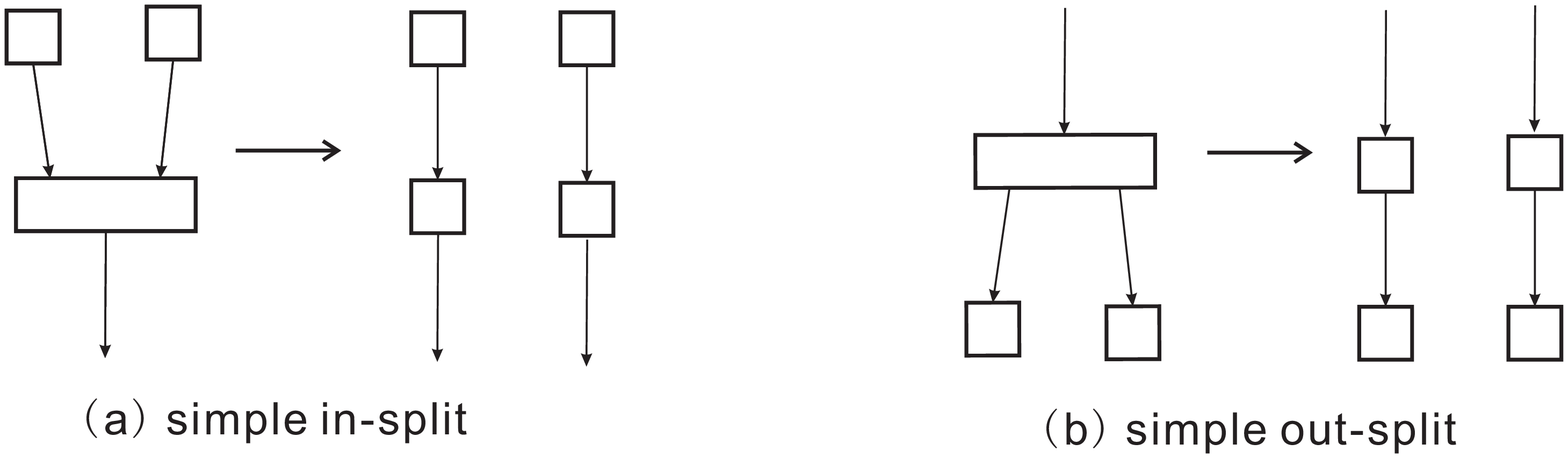}\\
  \caption{}\label{figure02}
  \end{center}
\end{figure}

\begin{figure}[htp]
\begin{center}
  \includegraphics[totalheight=7cm]{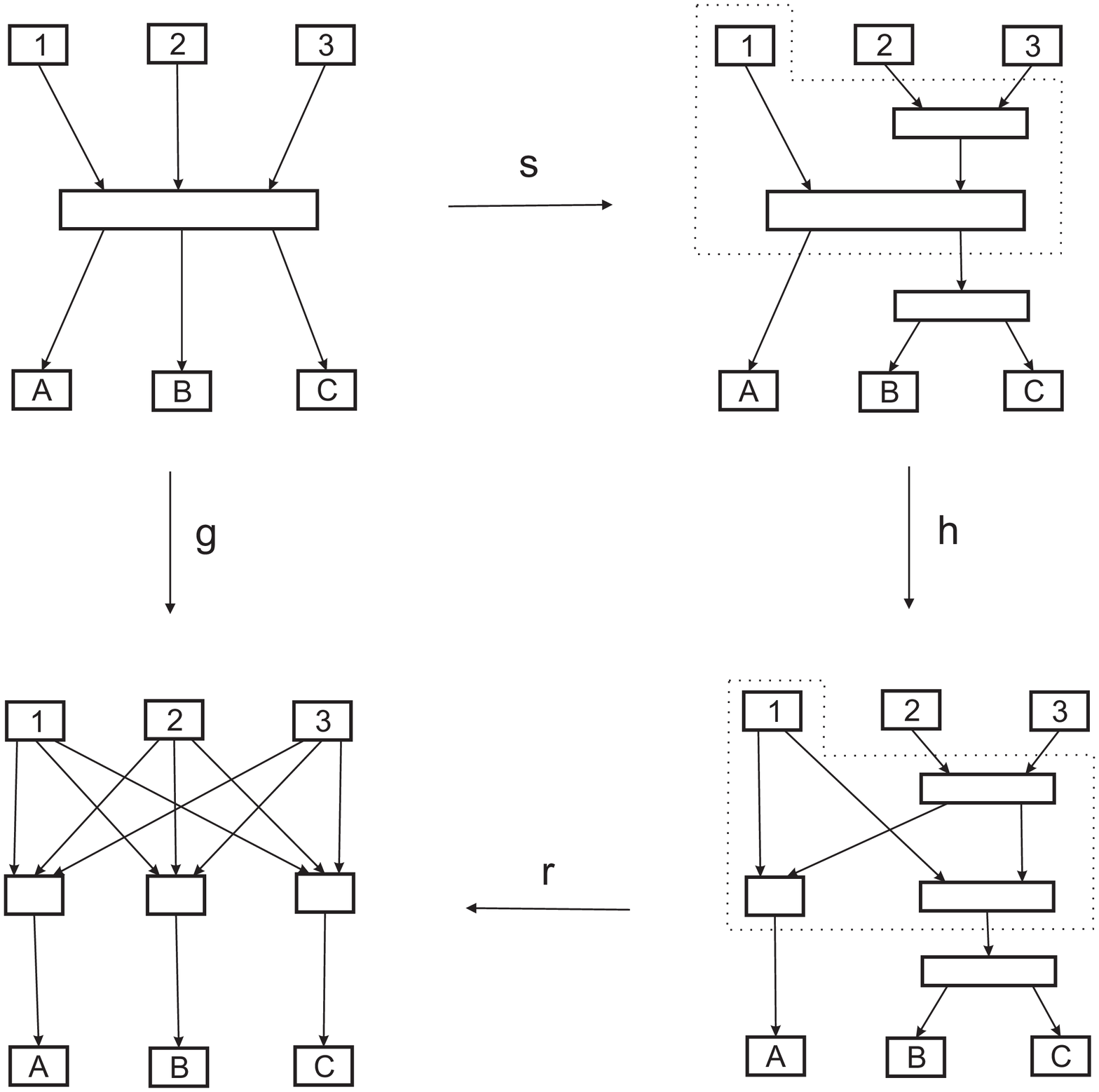}\\
  \caption{}\label{figure03}
  \end{center}
\end{figure}

By Lemma \ref{lemma3.1} and some observations, we can obtain
$\mathcal {C}_2$ from $\mathcal {C}_1$ by finitely many steps of
unstable direction divisions (or their converses), stable direction
divisions (or their converses) and adding parallel cross sections
(or their converses) such that, after each step, we obtain a new
cross-section of $\Lambda$. Each cross-section corresponds to a
Markov flowbox neighborhood (See Section \ref{section2.1} or
\cite{GHS}). By crushing the stable and unstable foliations of a
Markov flowbox neighborhood, we obtain an edge graph.

\begin{figure}[htp]
\begin{center}
  \includegraphics[totalheight=4cm]{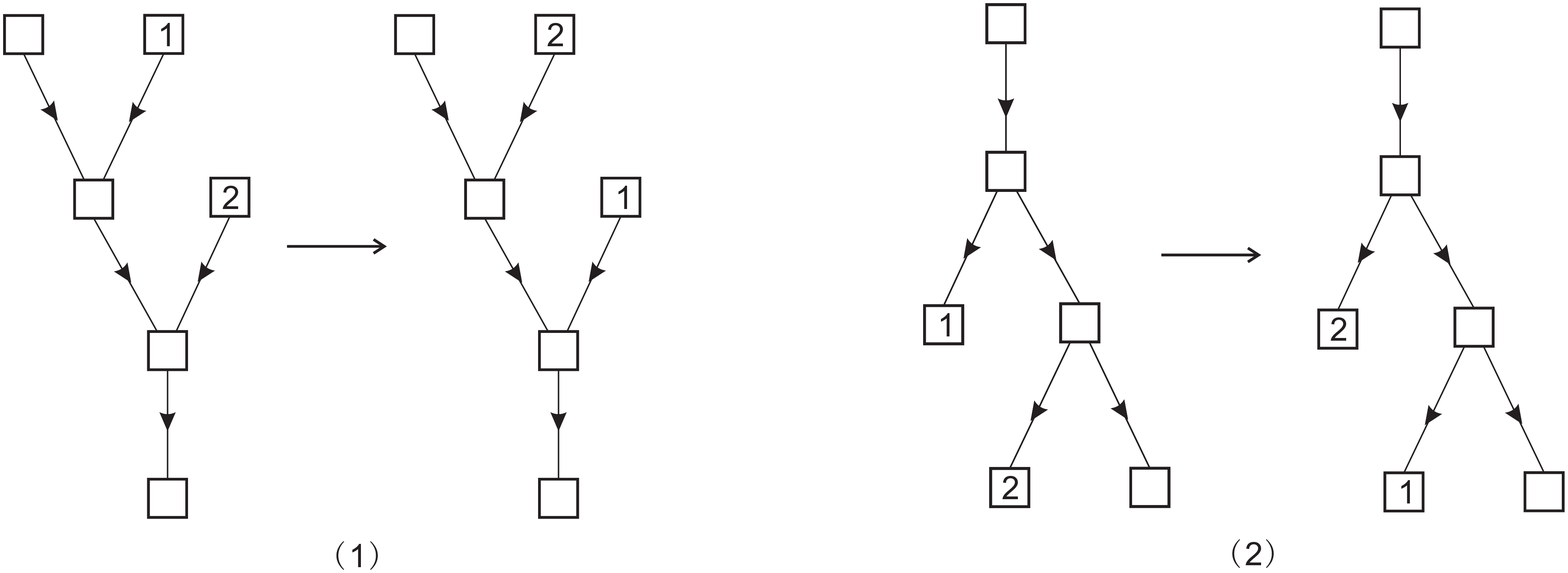}\\
  \caption{}\label{figure01}
  \end{center}
\end{figure}

\begin{figure}[htp]
\begin{center}
  \includegraphics[totalheight=6cm]{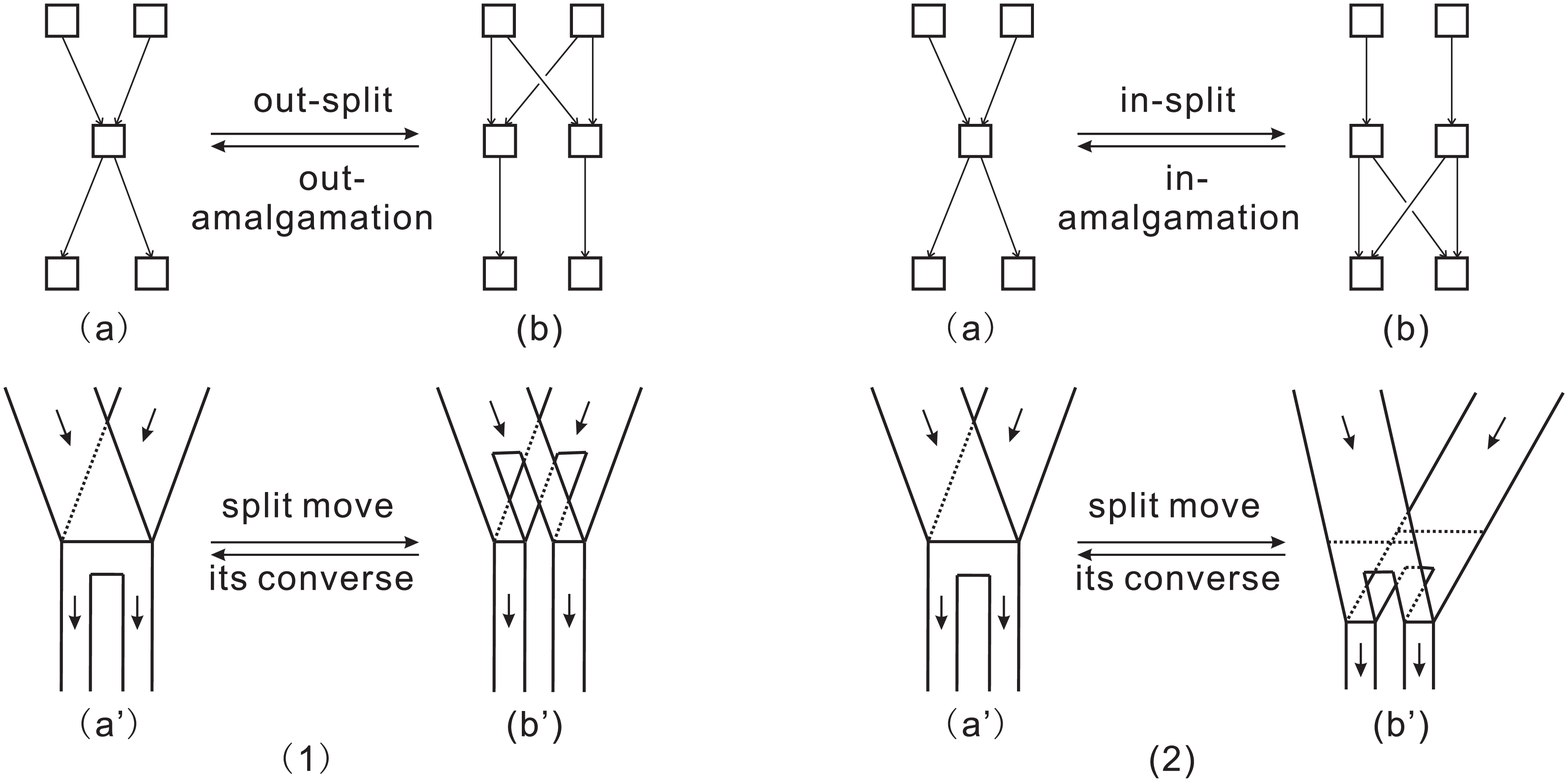}\\
  \caption{}\label{figure9}
  \end{center}
\end{figure}

Obviously, unstable direction division, stable direction division
and adding a parallel cross section on cross sections correspond to
in-split, out-split and expanding surgery on edge graphs
respectively. So, we can obtain $G_1$ from $G_2$ by $f_{n+1}\circ
g_{n}\circ f_{n}\circ ... \circ g_{1}\circ f_{1}$ ($n\in
\mathbb{N}$). Here each $f_i$ ($i=1,...,n+1$) is a composition of
finitely many steps of expanding surgery and/or its converse. and
each $g_i$ ($i=1,...,n+1$) is an out-split, an in-split or the
inverse of an out-split or in-split. For simplicity, we denote this
relation by $G_1 =f_{n+1}\circ g_{n}\circ f_{n}\circ ... \circ
g_{1}\circ f_{1} (G_2)$.

Each $g_i$ can be written as the composition $r_{i} \circ h_{i}
\circ s_{i}$ satisfying the following three conditions:
\begin{enumerate}
\item Both $s_{i}\circ f_{i}\circ ... \circ
g_{1}\circ f_{1}(G_2)$ and $h_{i} \circ s_{i}\circ f_{i}\circ ...
\circ g_{1}\circ f_{1}(G_2)$ are template edge graphs.
\item Both $r_i$ and $s_i$ are composed of finitely many steps of
simple splits (See Figure \ref{figure02}), expanding surgeries
and/or their converses.
\item $h_i$ is an out-split, an in-split or the inverse of an out-split or in-split
between two template edge graphs.
\end{enumerate}

Figure \ref{figure03} shows  a decomposition of an out-split. In a
similar way
 as Figure \ref{figure03} shows, each $g_i$ admits such a decomposition.

Set $t_1 =s_{1} \circ f_1$, $t_i = s_i \circ f_i \circ r_{i-1}$
($i=2,...,n$) and $t_{i+1}=f_{n+1}\circ r_n$. Then $G_1
=t_{n+1}\circ h_{n}\circ t_{n}\circ ... \circ h_{1}\circ t_{1}
(G_2)$. Here each $t_i$ ($i=1,...,n+1$) is a composition of finitely
many steps of simple splits, expanding surgeries and/or their
converses between two template edge graphs. It is easy to show that
each $t_i$ is a composition of the surgeries in Figure
\ref{figure01}, expanding surgeries and the converses of expanding
surgeries. The local structure of each $h_i$ is shown in Figure
\ref{figure9} (see also Figure \ref{figure03}).

Obviously $t_{n+1}\circ h_{n}\circ t_{n}\circ ...\circ h_{1}\circ
t_{1}: G_2 \rightarrow G_1$ provides the corresponding
transformation from $T_2$ to $T_1$, i.e., $\mathcal {T}_{n+1}\circ
\mathcal {H}_{n}\circ \mathcal {T}_{n}\circ ... \circ \mathcal
{H}_{1}\circ \mathcal {T}_{1}: T_2 \rightarrow T_1$. Here $\mathcal
{T}_i$ and $\mathcal {H}_j$ correspond to $t_i$ and $h_j$
respectively. By the constructions of $t_i$ and $h_i$, each
$\mathcal{T}_i$ corresponds to a composition of finitely many steps
of slide moves and each $\mathcal{H}_i$ corresponds to a split move
or the converse of a split move (see Figure \ref{figure9}).
Therefore, $T_1$ and $T_2$ can be exchanged by a finite sequence of
template moves.
\end{proof}

 \begin{theorem}\label{theorem3.3}
 Let $T_{1}$ and $T_{2}$ be two templates and $[\varphi^{T_1}_{t},
 K_{T_1}]$ and $[\varphi^{T_2}_{t}, K_{T_2}]$ be the germs  determined by $T_1$
 and $T_2$ respectively. Then $[\varphi^{T_1}_{t}, K_{T_1}]=[\varphi^{T_2}_{t}, K_{T_2}]$
 if and only if $T_{1}$ and $T_{2}$ can be exchanged by applying a finite sequence
 of template moves.
 \end{theorem}

\begin{proof}

\emph{The sufficiency}. It suffices to consider only one template
move and there are the following two possible cases:
\begin{enumerate}
\item $T_1$ is transformed into $T_2$ by a slide move;
\item $T_1$ is transformed into $T_2$ by a split move.
\end{enumerate}

By the construction of a thickened template, it is easy to show that
$(T_{1},\varphi^{T_1}_{t})$ is topologically equivalent to
$(T_{2},\varphi^{T_2}_{t})$ in case (1). Therefore, in this case,
$[\varphi^{T_1}_{t}, K_{T_1}]=[\varphi^{T_2}_{t}, K_{T_2}]$.

In case (2), there exists a template $T_{1}'\subset T_{1}$ such that
$T_{1}'=T_2$. So $(\overline{T_{1}'},\varphi^{T_{1}'}_{t})$ is
topologically equivalent to $(\overline{T_{2}},\varphi^{T_{2}}_{t})$
and $K_{T_1}\subset\overline{T_{1}'}\subset \overline{T_1}$.
Therefore, $[\varphi^{T_1}_{t}, K_{T_1}]=[\varphi^{T_2}_{t},
K_{T_2}]$.

\emph{The necessity}.
 If $[\varphi^{T_1}_{t},
K_{T_1}]=[\varphi^{T_2}_{t}, K_{T_2}]$, by the definition of germ,
there exists $U_1$ and $U_2$ which are the neighborhoods of
$K_{T_1}$ and $K_{T_2}$ respectively such that:
\begin{enumerate}
\item $K_{T_i}\subset U_i\subset \overline{T_i}$ for $i=1,2$;
\item $(U_1,\varphi^{T_1}_{t}|_{U_1})$ is topologically equivalent to
$(U_2,\varphi^{T_2}_{t}|_{U_2})$.
\end{enumerate}

Therefore, we can construct templates $T_{1}^{1}\subset U_1$ for
$(U_1,\varphi^{T_1}_{t}|_{U_1})$ and $T_{2}^{2}\subset U_2$ for
$(U_2,\varphi^{T_2}_{t}|_{U_2})$ such that $T_{1}^{1}=T_{2}^{2}$. On
the other hand, by Lemma \ref{lemma3.2}, $T_i$ and $T_{i}^{i}$
($i=1,2$) can be exchanged by applying a finite sequence of template
moves. Therefore, $T_1$ and $T_2$ can be exchanged by applying a
finite sequence of template moves.
\end{proof}

\section{Template, Model of Germ and Filtrating Neighborhood}

One may attach  a 2-handle along a dividing curve $c$ of a thickened
template $\overline{T}$ as shown in  Figure \ref{figure11}. Here a
2-handle is a cylinder $D^{2}\times [0,1]$ endowed with the vector
field $\frac{\partial}{\partial t}$. The first figure in Figure 11
is a cross-section of $\overline{T}$ along $c$.

\begin{figure}[htp]
\begin{center}
  \includegraphics[totalheight=4.5cm]{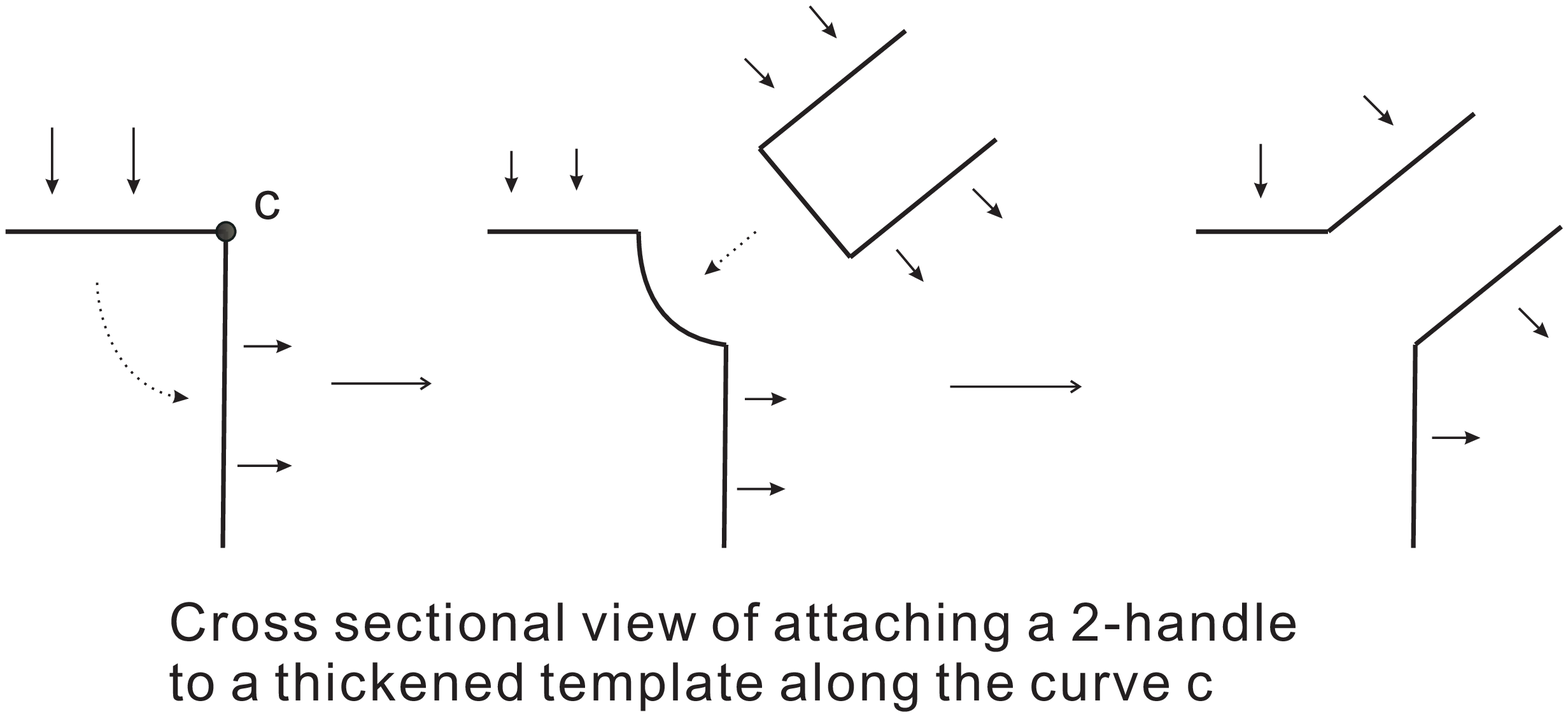}\\
  \caption{}\label{figure11}
  \end{center}
\end{figure}

If we repeat the above process for each dividing curve of
$\overline{T}$, we obtain a  compact 3-manifold $\overline{N}$ with
the flow $\overline{\varphi_{t}}$. Here $\overline{\varphi_{t}}$ is
transverse to $\partial \overline{N}$. Obviously, $K_{T}\subset
\overline{T}\subset \overline{N}$ is the maximal invariant set for
$\overline{\varphi_{t}}$ in $\overline{N}$.

In \cite{Fr}, G. Frank described $W^{s}(\varphi_t)\cap \partial
\overline{T}$. We restate his result in the following:

\begin{theorem}\label{theorem4.1}
$W^{s}(\varphi_t)\cap \partial \overline{T}=W^{s}(\varphi_{t})\cap
X$ where $X$ is the entrance set of $\partial \overline{T}$. It is a
lamination $\mathcal {F}$ on $X$ which locally is homeomorphic to
$\mathcal {C} \times \mathbb{R}$. Here $\mathcal {C}$ is a Cantor
set. In $\mathcal {F}$, there are only a finite number of circles.
For any $x\in \mathcal {F}$, ${x}\times t$ approaches a circle
asymptotically as $t\rightarrow +\infty$ (or $t\rightarrow
-\infty$). $X-\mathcal {F}$ has infinite connected components. Each
connected component of $X-\mathcal {F}$ is homeomorphic to the
interior of $D^2$.
\end{theorem}

\begin{remark}
The statement ``Each connected ...$D^2$" in Theorem \ref{theorem4.1}
is not included in the original theorem of G. Frank. But it is easy
to prove this fact by the arguments in his paper \cite{Fr}.
\end{remark}

 \begin{theorem}\label{theorem4.3}
Let $\overline{T}$ be a thickened template and let $\overline{N}$ be
obtained by attaching 2-handles to each dividing curve. Then
$(\overline{\varphi_{t}},\overline{N})$ is the model of the germ
$[\varphi^{T}_{t}, K_{T}]$.
 \end{theorem}

 \begin{proof}
 $(\overline{\varphi_{t}},\overline{N})$ obviously satisfies (1) and (3) in the
 definition of the model of a germ. Therefore, it is sufficient to
 prove that $(\overline{\varphi_{t}},\overline{N})$ satisfies (2) of the
 definition.

 For any circle $c$ in $\partial_{1}\overline{N}$ which is disjoint
from $W^{s}(\overline{\varphi_{t}})$, at least one of the following
three situations occurs.
 Here $\partial_{1}\overline{N}$ is the entrance set of $(\overline{\varphi_{t}},\overline{N})$.

 Case 1: $c\subset D^{2}\times0 \sqcup D^{2}\times1$ for some attached
 2-handle $D^{2}\times [0,1]$. Then obviously $c$ bounds a disk in
 $(D^{2}\times0 \sqcup D^{2}\times1)\subset
 \partial_{1}\overline{N}$. It is easy to see that the disk
 is disjoint from $W^{s}(\overline{\varphi_{t}})$.

 Case 2: $c\subset X\subset \partial_{1}\overline{N}$ and
 $c\cap W^{s}(\overline{\varphi_{t}})=\emptyset$. Obviously $c\cap
 W^{s}(\varphi_{t}^{T})=\emptyset$. By Theorem \ref{theorem4.1}, $c$ is in the interior of a
 disk which belongs to $X$ and is
 disjoint from $W^{s}(\overline{\varphi_{t}})$.
 Therefore, $c$ bounds a disk in $X\subset
 \partial_{1}\overline{N}$. Moreover, the disk
 is disjoint from $W^{s}(\overline{\varphi_{t}})$.

 Case 3: $c\cap X\neq\emptyset$ and $c\cap \{2-handles\}
 \neq\emptyset$. In this case, there exists an isotopy map $F_{t}(x):S^1\rightarrow
 \partial_{1}\overline{N}$ ($t\in [0,1]$) such that
$F_{0}(S^1)=c$, $F_{1}(S^1)=c_1$ and $F_{[0,1]}(S^1)\cap
W^{s}(\overline{\varphi_{t}})=\emptyset$, where $c_1$ is a simple
closed curve in $X$.

As in Case 2, $c_1$ bounds a disk $D_1\subset \partial \overline{N}$
which is disjoint from $W^{s}(\overline{\varphi_{t}})$. Let
$D=F_{[0,1]}(S^1)\cup D_1$ which is also a disk in $\partial_1
\overline{N}$. It is easy to check that $\partial D =c$ and $D\cap
W^{s}(\overline{\varphi_{t}})=\emptyset$.
 \end{proof}

Let $\Sigma_{g,k}$ be a genus $g$ orientable surface with $k$
punctures. We endow $\Sigma_{g,k}\times [0,1]$ with the vector field
$\frac{\partial}{\partial t}$. Let $\overline{T}$ be a thickened
template and denote the set of dividing curves by set
$C=\{c_1,c_2,...,c_n\}$ where each $c_i$ $(i=1,...,n)$ is a simple
closed curve. Consider a partition of $C$, i.e., $C=C_1 \sqcup C_2
\sqcup ... \sqcup C_t$ satisfying $C_j\subset C$ $(j=1,...,t)$ and
$C_i \cap C_j =\emptyset$ ($i\neq j$). Denote the number of the
connected components of $C_i$ by $k_i$. One may attach
$\Sigma_{g_i,k_i}\times [0,1]$ to $\overline{T}$  along $C_i$
similar to attaching 2-handle.  By Proposition \ref{proposition2.11}
and Theorem \ref{theorem4.3}, we have the following theorem.

\begin{theorem}\label{theorem4.4}
If $U$ is a filtrating neighborhood with a basic set modeled by $T$,
then there exists a partition $C_1 \sqcup C_2 \sqcup ... \sqcup C_t$
of $C$ such that $U$ is topologically equivalent to $(\phi_t,N)$.
Here $(\phi_t,N)$ is obtained by attaching $\Sigma_{g_i,k_i}\times
[0,1]$ to $\overline{T}$  along $C_i$ ($i=1,...,t$) for some
nonnegative integers $g_i$ $(i=1,...,t)$.
\end{theorem}

\begin{theorem}\label{theorem4.5}
Let $M=M'\sharp mS^1 \times S^2$ where $M'$ is prime to $S^1 \times
S^2$. Suppose $\phi_t$ is an NS flow on $M$ with a basic set
$\Lambda$ modeled by a template $T$. As Theorem \ref{theorem4.4}
shows, the filtrating neighborhood $U$ of $\Lambda$ is obtained by
attaching $\Sigma_{g_i,k_i}\times [0,1]$ $(i=1,...,t)$ to
$\overline{T}$ along dividing curves set $C$ of $\overline{T}$. Then
for any $i\in\{1,...,t\}$, we have $g_i\leq m+1$ and
\begin{enumerate}
\item  $k_i \leq 4m-3g_i +3$ for $g_i\leq m$;
\item  $k_i\leq m+1$ for $g_i=m+1$.
\end{enumerate}
\end{theorem}

\begin{proof}
By Theorem 1 of \cite{Yu2}, the genus of any connected component of
$\partial U$ is no more than $m+1$, hence $g_{i}\leq m+1$ for any
$i\in\{1,...,t\}$.

Suppose $\Sigma_{g_i,k_i} \times 0$ ($\Sigma_{g_i,k_i} \times 1$) is
attached to $\Sigma_0$ ($\Sigma_1$) which is a connected component
of $\partial U$. Under the assumption  $k_i\geq 4m-3g_i +4$, we
obtain the following two claims.

\textbf{Claim 1.} There exist at least $r(g_i)=2m-g_i +2$ simple
closed curves of $C_i \times 1$ such that these curves bound the
disks $D_1, D_2, ..., D_{r(g_i)}$ in $\Sigma_1$.

Otherwise, there exist $k=2m-2g_i +3$ simple closed curves in $C_i
\times 1$ bound some surfaces $S_1, S_2,.., S_s$ in $\Sigma_1$ such
that none of them is homeomorphic to a disk. Without loss of
generality, we suppose these simple closed curves are $c_1, ...,
c_k$. If $\partial S_j$ ($j=1,..,s$) is connected, then the genus of
$S_j$ is at least $1$.  Denote the number of the connected
components of $\partial S_j$ by $r_j$. If $\partial S_j$ isn't
connected , the genus of $(\Sigma_1 -S_j)$ is at most $m+1-r_j +1$.
Suppose none of $\partial S_1,...,\partial S_{\lambda}$
($\lambda\leq s$) is connected and each of $\partial S_{\lambda+1},
...,
\partial S_s$ is connected. It is easy to show that the genus of
$\Sigma_{g_i,k_i}$ is at most
\begin{equation}
\begin{split}
&m+1-\sum_{j=1}^{\lambda} (r_j -1) -(s-\lambda)\\
&=m+1-k+\lambda \quad (\text{since } \sum_{j=1}^{\lambda} r_j+(s-\lambda)=k)\\
&=2g_i -m-2+\lambda \quad (\text{since } k=2m-2g_i +3).
\end{split}\nonumber
\end{equation}
Noting that $r_j \geq 2$ and $\sum_{j=1}^{\lambda}
r_j+(s-\lambda)=k=2m-2g_i +3$, we have $\lambda\leq m-g_i +1$ and
$2g_i -m-2+\lambda\leq 2g_i -m-2+(m-g_i +1) =g_i -1$. Therefore, the
genus of $\Sigma_{g_i,k_i}$ is at most $g_i -1$. It contradicts the
fact that the genus of $\Sigma_{g_i,k_i}$ is $g_i$.

\textbf{Claim 2.} The disks $D_1,...,D_{r(g_i)}$ together with the
backward trajectory of their boundaries and $\Sigma_0
-\Sigma_{g_i,k_i} \times 0$ ($0\leq g_i \leq m+1$) compose at least
$m+1$ 2-spheres.

Case 1: $g_i =m+1$.

By the assumption ``$k_i\geq 4m-3g_i +4$" and Claim 1, we have
$k_i\geq m +1$ and there exist at least $r(m+1)=m+1$ simple closed
curves of $C_i \times 1$ which bound disks in $\Sigma_1$. By Theorem
1 of \cite{Yu2}, the genus of $\Sigma_0$ and $\Sigma_1$ are no more
than $m+1$, so the backward trajectory of the boundaries of
$D_1,...,D_{r(g_i)}$ also bounds disks in $\Sigma_0$. Therefore, if
$g_i =m+1$, Claim 2 is true.

 Case 2: $g_i \leq m$ and $C_i \times 0$ bound some genus $0$ surfaces $F_1, F_2,..., F_{\mu}$
 in $\Sigma_0$.

 Suppose $\partial F_i$ ($i=1,..,\mu$) has $r_i$ ($r_i\geq 1$)
 connected
 components. Obviously $g_i =m+1-\sum_{i=1}^{\mu} (r_i
 -1)=m+\mu+1-\sum_{i=1}^{\mu}
 r_i$. By Claim 1, there  exist at least $r(g_i)=2m-g_i
 +2=2m-(m+\mu+1)+\sum_{i=1}^{\mu}
 r_i +2=(m+1)+\sum_{i=1}^{\mu} (r_i -1)$ simple closed curves of $C_i \times 1$ bound disks $D_1, D_2, ...,
D_{r(g_i)}$ in $\Sigma_1$. Since $\partial F_i$ ($i=1,..,\mu$) has
$r_i$ ($r_i\geq 1$) connected
 components, the genus of $F_i$ is $0$ and $r(g_i)=(m+1)+\sum_{i=1}^{\mu} (r_i
 -1)$. Therefore, the backward trajectory of these simple closed curves (the
boundaries of $D_1, D_2, ..., D_{r(g_i)}$) bounds at least $m+1$
disks in $\Sigma_0$. Therefore, in this case, Claim 2 is also true.

Case 3: $g_i \leq m$ and there exists $\{c_1 \times 0,...,c_k \times
0\}\subset C_i \times 0$ bounds a genus $g$ ($g\geq 1$) surface in
$\Sigma_0$.

Attaching a surface $\Sigma_{g,k}$ to $\Sigma_{g_i,k_i}$ along its
boundary, we have a surface $\Sigma'$. The genus of $\Sigma'$ is
$g_i +g+ k-1$ and $\partial \Sigma'$ has at least $4m-3g_i + 4-k$
components. By Claim 1, there exist at least $r(g_i)=2m-g_i +2$
simple closed curves of $C_i \times 1$ which bound disks in
$\Sigma_1$, so there exist at least $2m-g_i +2-k$ simple closed
curves of $\partial \Sigma' \times 1$ which bound disks in
$\Sigma_1$.

Since the genus of $\Sigma'$ is $g_i +g+ k-1$, $\partial \Sigma'$
has at least $4m-3g_i + 4-k$ ($> 4m-3(g_i +g+ k-1)+4$) components
and there exist at least $2m-g_i +2-k$ ($> 2m-(g_i +g+ k-1)+2$)
simple closed curves of $\partial \Sigma' \times 1$ which bound
disks in $\Sigma_1$. By induction on the number of the connected
components of $\Sigma_0 -\Sigma \times 0$ whose genus is no less
than $1$, it can always be reduced to a situation similar to $g_i
\leq m+1$ or a situation in which $g_i \leq m$ and $C_i \times 0$
bound some genus $0$ surfaces. Here $\Sigma \times 0 \subset
\Sigma_0$. By the previous discussions, in these two situations,
Claim 2 is true. Therefore, Claim 2 is always true.

Now, let's return to the proof of the theorem. Let
$S_1^2,...,S_{m+1}^2$ be $m+1$ 2-spheres in Claim 2, since
$M=M'\sharp mS^1 \times S^2$ where $M'$ is prime to $S^1 \times
S^2$, there exists $r\in \mathbb{N}$ satisfying $r\leq m+1$ such
that $S=S_1^2 \cup S_2^2 \cup ... \cup S_r^2$ divides $M$ into two
connected components.

If $g_i \leq m$ and $k_i\geq 4m-3g_i +4 \geq m+4$, then $S$ borders
with $\overline{T}$ on both sides of $S$, the template $T$ can't be
connected. It is a contradiction. This means ``$k_i\geq 4m-3g_i +4$"
is impossible, therefore $k_i\leq 4m-3g_i +3$.

If $g_i =m+1$ and $k_i \geq m+2$, then $m+2 > 4m-3(m+1)+4= 4m-3g_i
+4$. Therefore,  Claim 2 is also true in this case. Since $k_i \geq
m+2$, $S$ borders with $\overline{T}$ on both sides of $S$, the
template $T$ can't be connected. It is a contradiction, so $k_i \leq
m+1$ when $g_i =m+1$.
\end{proof}

\begin{corollary}\label{corollary4.6}
There exist at most a finite number of filtrating neighborhoods (up
to topological equivalence) for a given closed orientable 3-manifold
$M$ and a template $T$ such that:
\begin{enumerate}
\item $T$ models the invariant sets of these filtrating
neighborhoods;
\item Everyone of these filtrating neighborhoods can be realized as a filtrating neighborhood of an NS
flow on $M$.
\end{enumerate}
\end{corollary}

\begin{proof}
Suppose $U$ is a filtrating neighborhood with basic set modeled by
$T$ and can be embedded into $M$. By Theorem \ref{theorem4.4}, $U$
is topologically equivalent to $(\phi_t,N)$ which is obtained by
attaching $\Sigma_{g_i,k_i}\times [0,1]$ to $\overline{T}$  along
$C_i$ ($i=1,...,t$) for a partition $C_1 \sqcup C_2 \sqcup ...
\sqcup C_t$ of $C$. Here $C$ is the dividing curves set of
$\overline{T}$ and $k_i$ is the number of connected components of
$C_i$. Suppose $M\cong M'\sharp m S^{1}\times S^{2}$ where $M'$ is
prime to $S^{1} \times S^{2}$. Obviously, the partition number of
$\{c_1,...,c_n\}$ is finite. By Theorem \ref{theorem4.5}, $g_i \leq
m+1$ for any $i\in \{1,2,...,t\}$.  We can obtain that $\sum_{i=1}^t
k_i =n$. Therefore, the number of all the possibilities of
$(\Sigma_{g_1,k_1},...,\Sigma_{g_t,k_t})$ is finite. This means  the
number of all the possibilities of the topological equivalence
classes of $U$ is finite.
\end{proof}

\begin{example}\label{example4.7}
Suppose $T=\mathcal {L}(1,1)$ is a template as Figure
\ref{figure12}-1 shows. It is a Lorenz like template, see \cite{GHS}
or \cite{Yu1}. The thickened template $\overline{T}$ of $T$ is shown
in Figure \ref{figure12}-2 and Figure \ref{figure12}-3. $X$ and $Y$
are the exit set and the entrance set of $\partial \overline{T}$
respectively. $\{a,b,c\}$ is the dividing curves set of
$\overline{T}$.  Figure \ref{figure12}-4 shows $X\cup \{a,b,c\}$.

\begin{figure}[htp]
\begin{center}
  \includegraphics[totalheight=6cm]{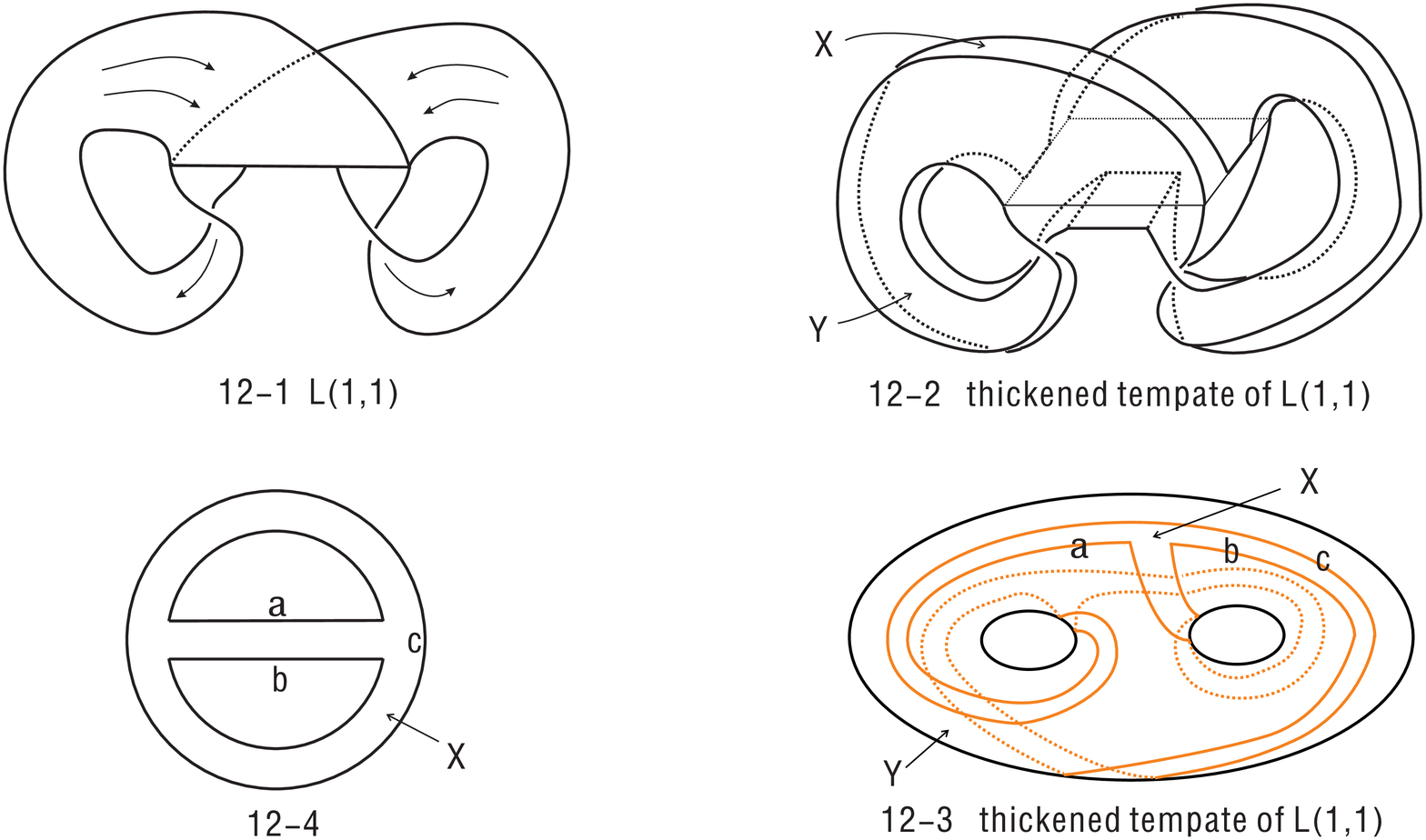}\\
  \caption{}\label{figure12}
  \end{center}
\end{figure}

Let $M$ be a closed orientable irreducible 3-manifold and
$(\phi_t,N)$ be a filtrating neighborhood with invariant set modeled
by $T$ which can be realized as a filtrating neighborhood of an NS
flow on $M$. Noting that $a$ and $b$ are symmetric in $\overline{T}$
and $M$ is prime to $S^1 \times S^2$, by a proof similar to that of
Corollary \ref{corollary4.6},
 we know that $(\phi_t,N)$ may be obtained by one of the following procedures:

\begin{enumerate}
\item attaching $\Sigma_{0,2}\times [0,1]$ to $\overline{T}$ along
$\{a,b\}$ and $\Sigma_{0,1}\times [0,1]$ to $\overline{T}$ along
$\{c\}$;
\item attaching $\Sigma_{0,2}\times [0,1]$ to $\overline{T}$ along
$\{a,c\}$ and $\Sigma_{0,1}\times [0,1]$ to $\overline{T}$ along
$\{b\}$;
\item attaching $\Sigma_{1,1}\times [0,1]$ to $\overline{T}$
along $\{a\}$ and two $\Sigma_{0,1}\times [0,1]$ to $\overline{T}$
along $\{b\}$ and $\{c\}$;
\item attaching $\Sigma_{1,1}\times [0,1]$ to $\overline{T}$
along $\{c\}$ and two $\Sigma_{0,1}\times [0,1]$ to $\overline{T}$
along $\{a\}$ and $\{b\}$.
\end{enumerate}

In case (1) and (2) above, by the proofs of Theorem 2 and Theorem 4
in \cite{Yu1}, it is easy to show that $N\cong L(3,1)\sharp
T^{2}\times [0,1]$. $M$ is irreducible and $N$ can be embedded into
$M$, therefore $M\cong L(3,1)$ and the interior of $N\cong
(L(3,1)-K)$.
 Here $K$ is a Hopf link in a three ball in $L(3,1)$.

In case (3) and (4) above, also  by the proofs of Theorem 2 and
Theorem 4 in \cite{Yu1}, we have that $N\cong S^{1}\times D^{2}-V$.
Here $V$ is a small open neighborhood of a standard $(3,1)$ torus
knot in the interior of $S^{1}\times D^{2}$. Since $N$ can be
realized as a filtrating neighborhood of an NS flow on $S^{1}\times
D^{2}$ and any closed orientable 3-manifold can admit an NS flow
(\cite{Yu2}), $M$ may be any closed orientable irreducible
3-manifold. In particular, if $M\cong S^3$ and $N$ can be realized
as a filtrating neighborhood of an NS flow on $S^3$, by the proofs
of Theorem 2 and Theorem 4 in \cite{Yu1}, $S^{3}-N$ may be a small
open neighborhood of a two components link $L$ which is composed of
a trivial knot and a $(p,3)$ torus knot in the boundary of a solid
torus neighborhood of the trivial knot. Here $p$ is any integer such
that $p$ and 3 are coprime.
\end{example}

\begin{remark}
Example \ref{example4.7} tells us that, given a closed orientable
3-manifold $M$ and a template $T$, although there exist at most a
finite number of filtrating neighborhoods  with invariant set
modeled by $T$ such that each of them can be realized as a
filtrating neighborhood of an NS flow on $M$, there may be
infinitely many different ways to embed some filtrating neighborhood
of $T$ as a filtrating neighborhood of an NS flow on $M$.
\end{remark}

\section{Embedding Templates in NS Flows on 3-Manifolds}
To prove the main theorem (Theorem \ref{theorem5.5}) of this
section, we first recall a theorem due to J. Morgan \cite{Mo}:
\begin{theorem}\label{theorem5.1}
Suppose that $M$ is an orientable 3-manifold with
$X(M,\partial_{-}M)= 0$ where $\partial_{-}M$ is composed of some
connected components of $\partial M$ and $X(M,\partial_{-}M)$ is the
Euler number of $(M,\partial_{-}M)$. For some nonnegative integer
$k$, $(M\sharp k S^{1}\times S^{2},
\partial_- M)$ has an NMS flow which is transverse outside to
$\partial_- M$ and transverse inside to $\partial M-\partial_- M$.
\end{theorem}

\begin{lemma} \label{lemma5.2}
Let $W$ be a compact 3-manifold with boundary $\Sigma\cup S_1\cup
...\cup S_g$ where $\Sigma$ is a genus $(g+1)$ closed orientable
surface and each $S_i$ ($i=1,...,g$) is homeomorphic to a 2-shpere.
Then for some nonnegative integer $k$, there exists an NMS flow
$\phi_t$ on $W\sharp kS^{1}\times S^{2}$ such that $\phi_t$ is
transverse outside to $\partial(W\sharp kS^{1}\times
S^{2})=\Sigma\cup S_1\cup ...\cup S_g$.
\end{lemma}

\begin{proof}
Attaching $W$ to a copy of $W$ along $\partial W$, we obtain a
closed 3-manifold $M$. Obviously $X(M)=0=2X(W,\partial W)-X(\partial
W)$. On the other hand,
\begin{equation}
X(\partial M)=\sum_{i=1}^{g}X(s_i)+X(\Sigma)
             =2g+(2-2(g+1))
             =0.
\nonumber
\end{equation}
Then we have  $X(W,\partial W)=0$, hence $(W,\partial W)$ satisfies
the condition of Theorem \ref{theorem5.1}. By Theorem
\ref{theorem5.1}, for some nonnegative integer $k$, $W\sharp
kS^{1}\times S^{2}$ admits an NMS flow $\phi_t$ which is transverse
outside to $\partial W$.
\end{proof}

\begin{theorem} \label{theorem5.3} \textbf{(Meleshuk \cite{Me})}
Any template $T$ can model a basic set $\Lambda$ of a Smale flow
$\varphi_t$ on $S^3$.
\end{theorem}

\begin{lemma} \label{lemma5.4}
There exists an NS flow $f_t$ on $S^2 \times I$ such that:
\begin{enumerate}
\item $f_t$ is transverse inward to $S^2 \times 0$ and outward to $S^2 \times
1$;
\item there is no flowline which starts in $S^2 \times 0$ and terminates
in $S^2 \times 1$.
\end{enumerate}
\end{lemma}

\begin{proof}
In the proof of Lemma 3.4 of \cite{Yu2}, choosing $n=1$, we get a
Smale flow $\varphi_t$ on $S^3$ with two singularities: one is a
sink, the other is a source. Cutting two standard small
neighborhoods of the sink and the source respectively, we obtain an
NS flow $f_t$ on $S^2 \times I$ which satisfies (1) of Lemma
\ref{lemma5.4}. By the argument in the proof of Lemma 3.4 of
\cite{Yu2}, it is easy to prove that there is no flowline of $f_t$
which starts in $S^2 \times 0$ and terminates in $S^2 \times 1$.
\end{proof}

\begin{theorem} \label{theorem5.5}
Let $T$ be a template. Then there exists a positive integer $n$ such
that  $nS^1 \times S^2$ admits an NS flow $\psi_t$  with a basic set
$\Lambda$ modeled by $T$.
\end{theorem}
\begin{proof}
Let $(\phi_t,N)$ be a filtrating neighborhood of $\Lambda$ of
$\varphi_t$ on $S^3$ where $\varphi_t$ is the Smale flow in Theorem
\ref{theorem5.3}. Suppose $\partial
N=\Sigma_1^+\cup...\cup\Sigma_t^+\cup\Sigma_1^-\cup...\cup\Sigma_s^-$
where $\Sigma_1^+\cup...\cup\Sigma_t^+$ and
$\Sigma_1^-\cup...\cup\Sigma_s^-$ are the entrance set and exit set
of $\varphi_t$ on $N$ respectively.

Suppose
\begin{equation}
\begin{split}
&g_i^+>1, ~~~i=1,...,t_1\\
&g_i^+=0, ~~~i=t_1+1,..., t_1+t_2\\
&g_i^+=1, ~~~i=t_1+t_2+1,...,t
\end{split}\nonumber
\end{equation}

and
\begin{equation}
\begin{split}
&g_j^->1, ~~~j=1,...,s_1\\
&g_j^-=0, ~~~j=s_1+1,..., s_1+s_2\\
&g_j^-=1, ~~~j=s_1+s_2+1,...,s.
\end{split}\nonumber
\end{equation}

Here each of $t_1$, $t_2$, $t-t_1-t_2$, $s_1$, $s_2$ and $s-s_1-s_2$
is a nonnegative integer and $g_i^+$ ($g_j^-$) is the genus of
$\Sigma_i^+$ ($\Sigma_j^-$).

Since any embedded surface in $S^3$ is separable,
$\overline{S^3-N}=M_1^+\sqcup...\sqcup M_t^+ \sqcup
M_1^-\sqcup...\sqcup M_s^-$ where $\partial M_i^+=\Sigma_i^+$ and
$\partial M_j^-=\Sigma_j^-$. Cutting $g_i^+ -1$ open 3-balls in
$M_i^+$ $(i=1,...,t_1)$, we get a compact 3-manifold $W_i^+$ such
that $\partial W_i^+$ is composed of $\Sigma_i^+$ and $g_i^{+}-1$
2-spheres. Similarly, cutting $g_j^- -1$ open 3-balls in $M_j^-$
$(j=1,...,s_1)$, we get a compact 3-manifold $W_j^-$ such that
$\partial W_j^-$ is composed of $\Sigma_j^-$ and $g_j^- -1$
2-spheres. By Lemma \ref{lemma5.2}, there exist NMS flows
$\phi_i^+(t)$ $(i=1,...,t_1)$ and $\phi_j^-(t)$ $(j=1,...,s_1)$ on
$W_i^+\sharp m_i S^1 \times S^2$ and $W_j^- \sharp n_j S^1\times
S^2$ respectively such that $\phi_i^+(t)$ is transverse outside to
$\partial W_i^+$ and $\phi_j^-(t)$ is transverse inside to $\partial
W_j^-$.

When $t_1 +t_2 +1\leq i\leq t$ and $s_1 +s_2 +1\leq j\leq s$,
$\partial M_i^+\cong \partial M_j^-\cong T^2$.  $M_i^+$ $(M_j^-)$
can be embedded into $S^3$, so $M_i^+$ $(M_j^-)$ is a knot
complement. By Proposition 6.1 of J. Franks \cite{F1}, there exists
an NS flow $\phi_i^+(t)$ ($\phi_j^-(t)$) on $M_i^+$ ($M_j^-$) which
is transverse inside (outside) to $\partial W_i^+$ ($\partial
W_j^-$).

By attaching $(W_i^+\sharp m_i S^1 \times S^2, \phi_i^+(t))$
$(i=1,...,t_1)$, $(W_j^- \sharp n_j S^1\times S^2, \phi_j^-(t))$
$(j=1,...,s_1)$, $(M_i^+, \phi_i^+(t))$  $(i=t_1 +t_2 +1,..., t)$
and $(M_j^-, \phi_j^-(t))$ $(j=s_1 +s_2 +1,..., s)$ to $(N, \phi_t)$
standardly, we get an NS flow $\psi_t^1$ on $V$ where $V$ is
homeomorphic to $(\sum_{i=1}^{t_1} m_i +\sum_{j=1}^{s_1} n_j) S^1
\times S^2$ with $(\sum_{i=1}^{t_1} g_i^+ -t_1)+ (\sum_{j=1}^{s_1}
g_j^- -s_1) + t_2 + s_2$ punctures. By Poinc\'are-Hopf theorem, it
is easy to show that $\sum_{i=1}^{t_1} g_{i}^+ -t_1
+s_2=\sum_{j=1}^{s_1}g_j^- -s_1 + t_2 $ (denote this number by $r$).
So $(V,\psi_t ^1)$ is transverse outside to 2-spheres $S_1^+,...,
S_r^+$ and inside to 2-spheres $S_1^-,...,S_r^-$.

If either $S_i^+$ belongs to some $W_i^+\sharp m_i S^1 \times S^2$
or $S_i^-$ belongs to some $W_j^-\sharp n_j S^1 \times S^2$, we
attach $S_i^+$ to $S_i^-$. Otherwise, we attach  $(S^2 \times I,
f_t)$ in Lemma \ref{lemma5.4} to $(V,\psi_t ^1)$ such that $S^2
\times 0$ is attached to $S_i^+$ and $S^2 \times 1$ is attached to
$S_i^-$.

Thus we obtain a flow $\varphi_t$ on $n S^1 \times S^2$ where
$n=\sum_{i=1}^{t_1} m_i +\sum_{j=1}^{s_1} n_j+r$. Noting that (2) in
Lemma \ref{lemma5.4} ensures that there are no more chain recurrent
points except the chain recurrent sets which we have controlled, it
is easy to check that $\varphi_t$ is an NS flow.
\end{proof}

By a comparison of Theorem \ref{theorem5.5} and the proof of Lemma
3.6 in \cite{Yu2}, we have the following corollary.

\begin{corollary}
Let $T$ be a template. There exists a positive integer $n$  such
that, for any closed orientable 3-manifold $M'$, there exists an NS
flow $\psi_t$ on $M=M'\sharp nS^1 \times S^2$ with a basic set
$\Lambda$ modeled by $T$.
\end{corollary}

Furthermore, the following question is natural.

\begin{question}
Let $T$ be a template. Does there exist a positive integer $n$ such
that any closed orientable 3-manifold $M$ admits an NS flow with a
basic set modeled by $T$ if and only if $M=M'\sharp nS^1 \times
S^2$? Here $M'$ is any closed orientable 3-manifold.
\end{question}

\section*{Acknowledgments} The author would like to thank Hao Yin for
his many helpful suggestions and comments. The author also would
like to thank the referee for carefully reading this paper and
providing many helpful suggestions and comments.

\bibliographystyle{amsplain}

\end{document}